\documentclass[11pt, twoside]{article}

\usepackage{amssymb}
\usepackage{amsmath}
\usepackage{amsthm}
\usepackage{color}
\usepackage{graphicx}
\usepackage{float}

\allowdisplaybreaks	

\def\rr{{\mathbb R}}
\def\rn{{{\rr}^n}}
\def\zz{{\mathbb Z}}
\def\cc{{\mathbb C}}
\def\nn{{\mathbb N}}

\def\ca{{\mathcal A}}

\def\cf{{\mathcal F}}

\def\cm{{\mathcal M}}
\def\cn{{\mathcal N}}
\def\cp{{\mathcal P}}

\def\mr{{\mathcal R}}
\def\cs{{\mathcal S}}

\def\cx{{\mathcal X}}

\def\fz{\infty}
\def\az{\alpha}
\def\bz{\beta}

\def\bgz{{\Gamma}}

\def\lz{\lambda}
\def\blz{\Lambda}
\def\oz{{\omega}}

\def\tz{\theta}

\def\vz{\varphi}

\def\lf{\left}
\def\r{\right}

\def\hs{\hspace{0.25cm}}
\def\ls{\lesssim}

\def\ov{\overline}
\def\noz{\nonumber}
\def\wz{\widetilde}
\def\wh{\widehat}
\def\st{\subset}
\def\com{\complement}
\def\bh{\backslash}

\def\dist{\mathop\mathrm{\,dist\,}}
\def\supp{\mathop\mathrm{\,supp\,}}

\def\loc{\mathop\mathrm{\,loc\,}}
\def\div{\mathop\mathrm{div}}
\def\essinf{\mathop\mathrm{\,ess\,inf\,}}
\def\esup{\mathop\mathrm{\,ess\,sup\,}}

\def\vp{{L^{p(\cdot)}(\mathbb{R}^n)}}
\def\vhp{H_L^{p(\cdot)}(\mathbb{R}^n)}
\def\nhp{H_{\mathcal{N}_h}^{p(\cdot)}(\mathbb{R}^n)}
\def\size{(p(\cdot),\,M,\,\varepsilon)_L}
\def\mol{\mathbb{H}_{L,\,M}^{p(\cdot),\,\vez}(\mathbb{R}^n)}
\def\mhp{H_{L,\,M}^{p(\cdot),\,\vez}(\mathbb{R}^n)}
\def\test{\mathcal{M}_{p(\cdot),\,L}^{\varepsilon,\,M}(\mathbb{R}^n)}
\def\tez{\mathcal{M}_{p(\cdot),\,L}^{M,\,\ast}(\mathbb{R}^n)}
\def\bmo{{\rm BMO}_{p(\cdot),\,L}^M(\mathbb{R}^n)}
\def\dbmo{{\rm BMO}_{p(\cdot),\,L^\ast}^M(\mathbb{R}^n)}
\def\uz{\underline}
\def\rnn{\rr_+^{n+1}}
\def\dydt{\,\frac{dy\,dt}{t^{n+1}}}
\def\vez{\varepsilon}

\def\ujb{{U_j(B)}}

\theoremstyle{definition}

\newtheorem{theorem}{Theorem}[section]
\newtheorem{proposition}[theorem]{Proposition}
\newtheorem{lemma}[theorem]{Lemma}
\newtheorem{corollary}[theorem]{Corollary}
\newtheorem{definition}[theorem]{Definition}
\newtheorem{remark}[theorem]{Remark}
\newtheorem{assumption}[theorem]{Assumption}

\numberwithin{equation}{section}

\begin{document}

\title{\bf\Large Variable Hardy Spaces Associated with Operators
Satisfying Davies-Gaffney Estimates
\footnotetext{\hspace{-0.35cm} 2010 {\it
Mathematics Subject Classification}. Primary 42B30; Secondary 42B35, 42B25, 35J15, 47B06.
\endgraf {\it Key words and phrases}. second order divergence form elliptic operator,
Davies-Gaffney estimate, variable Hardy space, square function, maximal function,
molecule, Riesz transform.
\endgraf Dachun Yang is supported by the National
Natural Science Foundation of China (Grant Nos. 11571039, 11671185 and 11361020).
Ciqiang Zhuo is supported by the Construct Program of the Key Discipline in
Hunan Province of China. }}

\author{Dachun Yang, Junqiang Zhang\,\footnote{Corresponding author / January 7, 2017.}\ \ and Ciqiang Zhuo}
\date{ }
\maketitle

\vspace{-0.8cm}

\begin{center}
\begin{minipage}{13.8cm}
{\small {\bf Abstract}\quad
Let $L$ be a one-to-one operator of type $\omega$ in $L^2(\mathbb{R}^n)$, with $\omega\in[0,\,\pi/2)$,
which has a bounded holomorphic functional calculus and satisfies the Davies-Gaffney estimates.
Let $p(\cdot):\ \mathbb{R}^n\to(0,\,1]$ be a variable exponent function
satisfying the globally log-H\"{o}lder continuous condition.
In this article, the authors introduce the variable Hardy space $H^{p(\cdot)}_L(\mathbb{R}^n)$
associated with $L$. By means of variable tent spaces, the authors establish
the molecular characterization of $H^{p(\cdot)}_L(\mathbb{R}^n)$.
Then the authors show that the dual space of $H^{p(\cdot)}_L(\mathbb{R}^n)$ is the BMO-type space
${\rm BMO}_{p(\cdot),\,L^\ast}(\mathbb{R}^n)$, where $L^\ast$ denotes the adjoint operator of $L$.
In particular, when $L$ is the second order divergence form elliptic operator
with complex bounded measurable coefficients, the authors obtain the non-tangential
maximal function characterization of $H^{p(\cdot)}_L(\mathbb{R}^n)$ and
show that the fractional integral $L^{-\alpha}$ for $\alpha\in(0,\,\frac12]$
is bounded from $H_L^{p(\cdot)}(\mathbb{R}^n)$ to $H_L^{q(\cdot)}(\mathbb{R}^n)$
with $\frac1{p(\cdot)}-\frac1{q(\cdot)}=\frac{2\alpha}{n}$ and
the Riesz transform $\nabla L^{-1/2}$ is bounded from $H^{p(\cdot)}_L(\mathbb{R}^n)$
to the variable Hardy space $H^{p(\cdot)}(\mathbb{R}^n)$.
}
\end{minipage}
\end{center}


\vspace{0.2cm}

\section{Introduction\label{s1}}
\hskip\parindent
The variable Lebesgue space $\vp$ is a generalization of classical Lebesgue spaces,
via replacing the constant exponent $p$ by a variable exponent function
$p(\cdot):\ \rn\to(0,\,\fz)$, which consists of all measurable functions $f$
such that, for some $\lz\in(0,\,\fz)$,
\begin{equation}\label{eq-vl}
\int_\rn \lf[\frac{|f(x)|}{\lz}\r]^{p(x)}\,dx<\fz.
\end{equation}
The study of variable Lebesgue spaces originated
from Orlicz \cite{or31} in 1931, which were further developed by Nakano \cite{na50,na51}.
The next major step in the investigation of variable function
spaces was made in the article of Kov\'{a}\v{c}ik and R\'{a}kosn\'{\i}k \cite{kr91} in 1991.
Since then, the interest in variable function spaces has increased steadily.
Nowadays these variable function spaces have been widely used in various
analysis branches, for example, in harmonic analysis \cite{cr03,cf13,die04,dhhr11},
in fluid dynamics \cite{acm02,ruz00}, in image processing \cite{cgzl08},
in partial differential equations and variational calculus \cite{acm05,hhl08,su09}.

Recently, as a generalization of classical Hardy spaces,
Nakai and Sawano \cite{ns12} introduced variable Hardy spaces $H^{p(\cdot)}(\rn)$,
established their atomic characterizations and
investigated their dual spaces.
Independently, Cruz-Uribe and Wang \cite{cw14} also studied the
variable Hardy spaces $H^{p(\cdot)}(\rn)$ with $p(\cdot)$ satisfying some conditions slightly weaker than those used in \cite{ns12}.
As a sequel of \cite{ns12}, Sawano \cite{s13} sharpened the conclusion of the atomic
characterization of $H^{p(\cdot)}(\rn)$ in \cite{ns12}, which was used, in \cite{s13},
to establish the boundedness in $H^{p(\cdot)}(\rn)$ of the fractional integral operator and the commutators
generated by singular integral operators and BMO functions.
After that, Yang et al. \cite{yzn15,zyl14} established equivalent characterizations
of variable Hardy spaces via Riesz transforms and intrinsic square functions.

On the other hand, in recent years, there has been a lot of
attention paid to the study of function spaces, especially Hardy spaces
and BMO spaces, associated with various operators; see, for example,
\cite{adm05,amr08,bckyy13a,bckyy13,dl13,dy05,dy051,hm09,hmm11,jy10,zcjy15}.
Here, let us give a brief overview of this research direction.
First, Auscher et al. \cite{adm05}, and then Duong and Yan
\cite{dy05,dy051}, introduced Hardy and BMO spaces associated with an operator
$L$ whose heat kernel has a pointwise Gaussian upper bound.
Later, Hardy spaces associated with operators which satisfy the weaker conditions,
the so-called Davies-Gaffney type estimates, were treated in \cite{amr08,hlmmy11,hm09,hmm11}.
More precisely, Auscher et al. \cite{amr08} and Hofmann et al. \cite{hm09,hmm11}
treated Hardy spaces associated, respectively, with the Hodge Laplacian on a Riemannian
manifold equipped with a doubling measure, or with a second order divergence form
elliptic operator on $\rn$ with complex coefficients,
in which settings pointwise heat kernel bounds may fail.
Hofmann et al. \cite{hlmmy11} studied Hardy spaces associated with non-negative
self-adjoint operators satisfying the Davies-Gaffney estimates in
the general setting of a metric space with a doubling measure.
Then the weighted Hardy spaces associated with operators were also
considered in \cite{bd12,sy10}.
Recently, by introducing a notion of reinforced off-diagonal
estimates (see Remark \ref{rem-1x}(ii) below),
Bui et al. \cite{bckyy13} studied the weighted Hardy spaces
associated with non-negative self-adjoint operators satisfying
such estimates, which, in some sense, improve those results of \cite{bd12,sy10}
by extending the range of the considered weights.
To study the Hardy spaces associated with differential operators
on more general underlying spaces (for example, the Laplace-Beltrami operator on any Riemannian manifold
with a doubling measure),
Bui et al. \cite{bckyy13a} introduced Musielak-Orlicz-Hardy
spaces associated with operators satisfying reinforced off-diagonal estimates on
balls on a metric space with a doubling measure.
The notion of reinforced off-diagonal estimates on balls
(see Remark \ref{rem-1x}(iii) below)
was first introduced
in \cite{bckyy13a} by combining the ideas of the reinforced off-diagonal estimates
from \cite{bckyy13} and the off-diagonal estimates on balls from \cite{am07}.

Very recently, Yang and Zhuo \cite{yzo15} introduced variable Hardy spaces
associated with operators $L$ on $\rn$, denoted by $H^{p(\cdot)}_L(\rn)$, where
$p(\cdot):\ \rn\to(0,\,1]$
is a variable exponent function satisfying the globally
log-H\"{o}lder continuous condition
and $L$ is a linear operator on $L^2(\rn)$ which generates an
analytic semigroup $\{e^{-tL}\}_{t\geq 0}$ with kernels having pointwise upper bounds.
Moreover,  in \cite{yzo15},
the molecular characterization of $H^{p(\cdot)}_L(\rn)$ was established,
which was further applied to study the boundedness of the fractional integral associated
with $L$ on $H^{p(\cdot)}_L(\rn)$, and the dual space of $H^{p(\cdot)}_L(\rn)$ was also
investigated. Under an additional condition that $L$ is non-negative self-adjoint,
the atomic and several maximal function characterizations of $H_L^{p(\cdot)}(\rn)$ were
established in a recent article \cite{zy15}.

Motivated by \cite{hlmmy11,hm09,hmm11,yzo15}, in this article,
we consider the variable Hardy spaces $\vhp$ associated with a one-to-one operator $L$
of type $\omega$ in $L^2(\rn)$, with $\omega\in[0,\,\pi/2)$,
which has a bounded holomorphic functional calculus and satisfies the
Davies-Gaffney estimates, namely, Assumptions \ref{as-a} and \ref{as-b}
below of this article. We point out that
many operators satisfy these assumptions (see Remark \ref{rem Assu*} below).
Indeed, Assumption \ref{as-b} (the Davies-Gaffney estimates)
is weaker than the reinforced off-diagonal estimates from
\cite{bckyy13} and the reinforced off-diagonal estimates on balls
from \cite{bckyy13a} (see Remark \ref{rem-1x} below).
Under Assumptions \ref{as-a} and \ref{as-b}, we introduce the variable Hardy spaces $\vhp$
(see Definition \ref{def vhp} below). Then we establish their molecular characterizations
via variable tent spaces. By borrowing some ideas from \cite{hm09,jy10},
we further prove that the dual space of $\vhp$ is the BMO-type space
${\rm BMO}_{p(\cdot),\,L^\ast}(\rn)$, where $L^\ast$ denotes the adjoint operator of $L$.
In particular, when $L$ is the second order divergence form elliptic operator
with complex bounded measurable coefficients,
namely, $L:=-\div (A\nabla)$
(see \eqref{eq op} below for its definition),
we obtain the non-tangential maximal function characterization of $\vhp$
and establish the boundedness of the associated fractional integral
and Riesz transform on $\vhp$.

Compared with the function spaces with constant exponents,
a main difficulty appearing in the study on variable function spaces
exists in that the quasi-norm $\|\cdot\|_{\vp}$ has no explicit
and direct expression. Indeed, $\|\cdot\|_{\vp}$ is just the Minkowski functional
of a convex modular ball
$\{f\in\vp:\ \int_\rn |f(x)|^{p(x)}\,dx\le 1\}$
(see \eqref{eq norm} below), which makes many estimates become
very complicated. To overcome this difficulty, in this article, we borrow some ideas
from Sawano \cite{s13}, to be precise, slight variants of \cite[Lemmas 4.1 and 5.2]{s13}
(which are restated as Lemmas \ref{lem key} and \ref{lem 5.2} below).
The role of Lemma \ref{lem key} is to reduce some estimates in terms of
$L^{p(\cdot)}(\rn)$ norms of some series of functions into some estimates
in terms of $L^q(\rn)$ norms of some functions,
while Lemma \ref{lem 5.2} establishes some connection between $\|\cdot\|_{L^{p(\cdot)}(\rn)}$
and $\|\cdot\|_{L^{q(\cdot)}(\rn)}$ of infinite linear combinations of characteristic functions;
both lemmas play crucial roles in proving the main results of this article.
On the other hand, observe that the heat semigroup of the operators
considered in \cite{yzo15} has the pointwise upper bounds, while
the heat semigroup of the operators considered in this article
only satisfies some integral estimates; this difference leads to that the proofs
of main results of this article become more difficult and hence need some
subtler and more careful estimates, compared with those proofs of the corresponding
results in \cite{yzo15}.

This article is organized as follows.

In Section \ref{s2}, we first describe Assumptions \ref{as-a} and \ref{as-b}
imposed on the considered operator $L$ of this article.
Then we recall some notation and notions on the variable Lebesgue
space $\vp$ and give out the definition of the Hardy space $\vhp$ in terms of the square
function of the heat semigroup generated by $L$.

In Section \ref{s3}, we establish the molecular characterization of $\vhp$
(see Theorem \ref{thm 1} below), which is an immediate consequence of Propositions \ref{pro 1}
and \ref{pro 2} below. In particular, Proposition \ref{pro 1} shows
that the molecular Hardy space is a subspace of $\vhp$ and in its proof,
to overcome the difficulty caused by the variable $\|\cdot\|_{L^{p(\cdot)}(\rn)}$,
we need to apply Lemma \ref{lem key},
which is different from the proofs of
corresponding results of $H_L^p(\rn)$ established in \cite{bckyy13a,hm09,hmm11}.
Proposition \ref{pro 2} shows that $\vhp$ is a subspace of the molecular Hardy space and
its proof depends on the atomic decomposition of the variable tent space $T^{p(\cdot)}(\rn)$
from \cite{zyl14} (which is restated as Lemma \ref{lem 2} below) and also on Lemma \ref{lem 5}
which shows that an atom of $T^{p(\cdot)}(\rn)$ is a molecule of $\vhp$
under the projection operator $\pi_{M,\,L}$, with $M\in\nn$, defined in \eqref{eq pi} below.
To show Lemma \ref{lem 5}, we need to make full use of properties of
the Davies-Gaffney estimates from Assumption \ref{as-b},
since the heat semigroup $\{e^{-tL}\}_{t> 0}$ considered in this article has
no pointwise upper bounds, which is essentially different from that of \cite{yzo15}.

In Section \ref{s4}, by borrowing some ideas from \cite{hm09,hmm11,jy10},
we introduce the BMO-type space ${\rm BMO}^M_{p(\cdot),\,L^\ast}(\rn)$ with $M\in\nn$
(see Definition \ref{def bmo} below) and establish the duality between $\vhp$ and
${\rm BMO}^M_{p(\cdot),\,L^\ast}(\rn)$ (see Theorem \ref{thm 2} below).
To prove Theorem \ref{thm 2}, we need to first give out several properties
related to  ${\rm BMO}^M_{p(\cdot),\,L^\ast}(\rn)$
(see Proposition \ref{pro 4}, Remark \ref{rem 4.1}, Lemmas \ref{lem 4.1},
\ref{lem 4.2} and \ref{lem 4.3} below). The essential difficulty arising here is that
the quasi-norm $\|\cdot\|_{\vp}$, in general, has no property of the
translation invariance, namely, for any $z\in\rn$ and
ball $B(x,\,r)\st\rn$ with $x\in\rn$ and $r\in(0,\,\fz)$,
$\|\chi_{B(x,\,r)}\|_{\vp}$ may not be equal to $\|\chi_{B(x+z,\,r)}\|_{\vp}$.
To overcome this difficulty, we make full use of a slight variant of \cite[Lemma 2.6]{zyl14}
(see Lemma \ref{lem 3} below), which is different from the case that
 $p(\cdot)\equiv\,\mathrm{constant}\in(0,\,1]$ as in \cite{hm09,hmm11}.

In Section \ref{s5}, as an example of operators satisfying
Assumptions \ref{as-a} and \ref{as-b}, we consider the second order
divergence form elliptic operator
$L:=-\div (A\nabla)$
with complex bounded measurable coefficients.
In Subsection \ref{s5.1}, by making full use of
the divergence structure of $L$, we obtain the non-tangential maximal function characterization of $\vhp$
(see Theorem \ref{thm 3} below).
The proof of Theorem \ref{thm 3} mainly depends on the extrapolation theorem for $\vp$
(see \cite[Theorem 1.3]{cfmp06} or Lemma \ref{lem extro} below), which
reduces the proof of Theorem \ref{thm 3} to some inequality
in terms of the weighted Lebesgue space with constant exponent in \cite{bckyy13a}.
In Subsection \ref{s5.2}, as an application of the molecular characterization of $\vhp$
in Theorem \ref{thm 1}, we show that the fractional integral $L^{-\az}$
is bounded from $\vhp$ to $H_L^{q(\cdot)}(\rn)$
for $\az\in(0,\,\frac12]$ and $\frac1{p(\cdot)}-\frac1{q(\cdot)}=\frac{2\az}{n}$
(see Theorem \ref{thm 5} below).
The proof of Theorem \ref{thm 5} strongly depends on a slight variant of \cite[Lemma 5.2]{s13},
namely, Lemma \ref{lem 5.2} below. In Subsection \ref{s5.3}, by borrowing some ideas from \cite{bckyy13,jy11,zcjy15},
we prove that the Riesz transform $\nabla L^{-1/2}$
is bounded from $\vhp$ to $H^{p(\cdot)}(\rn)$ (see Theorem \ref{thm 6} below)
via the atomic characterization of $H^{p(\cdot)}(\rn)$ (see \cite[Theorem 1.1]{s13}).

We end this section by making some conventions on notation. Throughout this article,
we denote by $C$ a positive constant which is independent of the main parameters,
but it may vary from line to line.
We also use $C_{(\az, \bz,\ldots)}$ to denote a positive constant depending on
the parameters $\az$, $\bz$, $\dots$.
The \emph{symbol $f\ls g$} means that $f\le Cg$.
If $f\ls g$ and $g\ls f$, then we write $f\sim g$.
For any measurable subset $E$ of $\rn$, we denote by $E^\com$ the \emph{set $\rn\bh E$}
and by $\chi_E$ the characteristic function of $E$.
For any $a\in\mathbb{R}$, the \emph{symbol} $\lfloor a\rfloor$
denotes the largest integer $m$ such that $m\le a$.
Let $\nn:=\{1,\,2,\,\ldots\}$ and $\zz_+:=\nn\cup\{0\}$.
Let $\rnn:=\rn\times(0,\fz)$. For any $\az\in(0,\,\fz)$ and $x\in\rn$, define
\begin{equation}\label{eq bgz}
\bgz_\az(x):=\{(y,\,t)\in\rnn:\ |y-x|<\az t\}.
\end{equation}
If $\az=1$, we simply write $\bgz(x)$ instead of $\bgz_\az(x)$.

For any ball $B:=B(x_B,r_B)\st\rn$ with $x_B\in\rn$ and $r_B\in(0,\,\fz)$, $\az\in(0,\fz)$ and $j\in\nn$,
we let $\az B:=B(x_B,\az r_B)$,
\begin{align}\label{eq ujb}
U_0(B):=B\ \ \ \text{and}\ \ \ U_j(B):=(2^jB)\setminus (2^{j-1}B).
\end{align}
For any $p\in[1,\,\fz]$, $p'$ denotes its conjugate number,
namely, $1/p+1/p'=1$.

For any $r\in(0,\,\fz)$, denote by $L^r_{\loc}(\rn)$ the set of all
\emph{locally $r$-integrable functions} on $\rn$ and, for any measurable
set $E\st\rn$, let $L^r(E)$ be the set of all measurable functions $f$
on $E$ such that $\|f\|_{L^r(E)}:=[\int_E |f(x)|^r\,dx]^{1/r}<\fz.$

\section{Preliminaries\label{s2}}
\hskip\parindent
In this section, we first
describe some basic assumptions on the operator $L$ studied throughout this article.
Then we recall some notation and notions on variable Lebesgue spaces and introduce the variable Hardy spaces $H_{L}^{p(\cdot)}(\rn)$
associated with $L$.

\subsection{Two assumptions on the operator $L$\label{s2.2}}
\hskip\parindent
Before giving the assumptions on the operator $L$ studied in this article,
we first recall some knowledge about
bounded holomorphic functional calculi introduced by McIntosh \cite{m86}
(see also \cite{adm96}).

Let $\omega\in[0,\,\pi)$. The \emph{closed} and \emph{open $\omega$ sectors},
$S_{\oz}$ and $S_{\oz}^0$, are defined, respectively, by setting
\begin{equation*}
S_{\oz}:=\{z\in\mathbb{C}:\ |\arg z|\le\oz\}\cup\{0\}
\quad\mathrm{and}\quad
S^0_{\oz}:=\{z\in\mathbb{C}\setminus\{0\}:\ |\arg z|<\oz\}.
\end{equation*}
A closed and densely defined operator $T$ in $L^2(\rn)$ is said to be of \emph{type $\oz$} if
\begin{enumerate}
\item[(i)] the spectrum $\sigma(T)$ of $T$ is contained in $S_\oz$.

\item[(ii)] for any $\tz\in(\oz,\,\pi)$, there exists a positive constant $C_{(\tz)}$
such that, for any $z\in \mathbb{C}\setminus S_\tz$,
$$|z|\|(zI-T)^{-1}\|_{\mathcal{L}(L^2(\rn))}\le C_{(\tz)},$$
here and hereafter, $\mathcal{L}(L^2(\rn))$ denotes the set of all continuous
linear operators from $L^2(\rn)$ to itself and,
for any $S\in\mathcal{L}(L^2(\rn))$,
the operator norm of $S$ is denoted by $\|S\|_{\mathcal{L}(L^2(\rn))}$.
\end{enumerate}

For any $\mu\in(0,\,\pi)$, define
$$H_\fz(S_\mu^0):=\{f:\ S_\mu^0\to\mathbb{C}\ \text{is holomorphic and}\ \|f\|_{L^\fz(S_\mu^0)}<\fz\}$$
and
\begin{align*}
\Psi(S_\mu^0):=\lf\{f\in H_\fz(S_\mu^0):\ \exists\,\az,\,C\in(0,\,\fz)\ \text{such that}\
|f(z)|\le \frac{C|z|^\az}{1+|z|^{2\az}},\ \forall z\in S_\mu^0\r\}.
\end{align*}

For any $\oz\in[0,\,\pi)$, let $T$ be a one-to-one operator of type $\oz$ in $L^2(\rn)$.
For any $\psi\in\Psi(S_\mu^0)$ with $\mu\in(\oz,\,\pi)$,
the operator $\psi(T)\in\mathcal{L}(L^2(\rn))$ is defined by setting
\begin{equation}\label{eq function}
\psi(T):=\int_\gamma\psi(\xi)(\xi I-T)^{-1}\,d\xi,
\end{equation}
where $\gamma:=\{re^{i\nu}:\ r\in(0,\,\fz)\}\cup\{re^{-i\nu}:\ r\in(0,\,\fz)\}$,
$\nu\in(\omega,\,\mu)$, is a curve consisting of two rays parameterized anti-clockwise.
It is easy to see that the integral in \eqref{eq function} is absolutely
convergent in $L^2(\rn)$ and the definition of $\psi(T)$ is independent of
the choice of $\nu\in(\oz,\,\mu)$ (see \cite[Lecture 2]{adm96}).
It is well known that the above holomorphic functional calculus defined on
$\Psi(S_\mu^0)$ can be extended to $H_\fz(S_\mu^0)$ by a limiting procedure
(see \cite{m86}).
Let $0\le\omega<\mu<\pi$. Recall that the operator $T$ is said to have a
\emph{bounded holomorphic functional calculus} in $L^2(\rn)$ if
there exists a positive constant $C_{(\omega,\mu)}$, depending on $\omega$ and $\mu$,
such that, for any $\psi\in H_\fz(S_\mu^0)$,
\begin{equation}\label{e2.1x}
\|\psi(T)\|_{\mathcal{L}(L^2(\rn))}\le C_{(\omega,\mu)}\|\psi\|_{L^\fz(S_\mu^0)}.
\end{equation}
By \cite[Theorem F]{adm96}, we know that, if \eqref{e2.1x} holds true for some
$\mu\in (\omega,\pi)$, then it also holds true for all $\mu\in (\omega,\pi)$.

\begin{remark}\label{r-2.0}
Let $T$ be a one-to-one operator of type $\omega$ in $L^2(\rn)$ with $\oz\in[0,\,\pi/2)$.
Then it follows from \cite[Theorem 1.45]{ou05}
that $T$ generates a bounded holomorphic semigroup $\{e^{-zT}\}_{z\in S^0_{\pi/2-\oz}}$
on the open sector $S^0_{\pi/2-\oz}$.
\end{remark}

We now make the following two assumptions on the operator $L$, which are used
through the whole article.

\begin{assumption}\label{as-a}
$L$ is a one-to-one operator of type $\omega$ in $L^2(\rn)$, with $\oz\in[0,\,\pi/2)$,
and has a bounded holomorphic functional calculus.
\end{assumption}

\begin{assumption}\label{as-b}
The semigroup $\{e^{-tL}\}_{t>0}$ generated by $L$
satisfies the \emph{Davies-Gaffney estimates},
namely, there exist positive constants $C$ and $c$ such that, for any closed subsets
$E$ and $F$ of $\rn$ and $f\in L^2(\rn)$ with $\supp f\st E$,
\begin{align}\label{eq-dg}
\lf\|e^{-tL}(f)\r\|_{L^2(F)}\le Ce^{-c\frac{[\dist(E,\,F)]^2}{t}}\|f\|_{L^2(E)}.
\end{align}
Here and hereafter, for any subsets $E$ and $F$ of $\rn$,
$$\dist(E,\,F):=\inf\{|x-y|:\ x\in E,\,y\in F\}.$$
\end{assumption}

\begin{remark}\label{rem-1x}
\begin{enumerate}
\item[(i)] The notion of the Davies-Gaffney estimates (or the so-called $L^2$ off-diagonal estimates)
of the semigroup $\{e^{-tL}\}_{t>0}$ was first introduced by Gaffney
\cite{gaf59} and Davies \cite{dav95}, which serves as good substitutes
of the Gaussian upper bound of the associated heat kernel;
see also \cite{am07} and related references therein.

\item[(ii)] Let $L$ be a non-negative self-adjoint operator on $L^2(\rn)$
and $\{e^{-tL}\}_{t>0}$ be the analytic semigroup generated by $L$.
The reinforced off-diagonal estimates introduced in \cite{bckyy13}
is that there exists a constant $p_L\in[1,\,2)$ such that, for all
$p_L<p\le q<p_L'$,
$\{e^{-tL}\}_{t>0}$ satisfies $L^p-L^q$ \emph{off-diagonal estimates},
denoted by $e^{-tL}\in\mathcal{F}(L^p-L^q)$, namely,
there exist positive constants $C$ and $c$ such that, for all $t\in(0,\,\fz)$,
any closed subsets $E$ and $F$ of $\rn$ and $f\in L^p(\rn)$ with $\supp f\st E$,
\begin{align*}
\lf\|e^{-tL}(f)\r\|_{L^q(F)}\le C t^{-\frac{n}2(\frac1p-\frac1q)}
e^{-c\frac{[\dist(E,\,F)]^2}{t}}\|f\|_{L^p(E)},
\end{align*}
which is obviously stronger than Assumption \ref{as-b} in this article.

\item[(iii)] Let $(\cx,\,d)$ be a metric space with a doubling measure $\mu$
and $L$ be a one-to-one operator of type $\omega$ in $L^2(\cx)$ with $\omega\in(0,\,\pi/2)$.
The \emph{reinforced off-diagonal estimates on balls} introduced in \cite{bckyy13a}
is that there exist constants $p_L\in[1,\,2)$ and $q_L\in(2,\,\fz]$ such that,
for all $p_L<p\le q<q_L$, $\{e^{-tL}\}_{t>0}$ satisfies $L^p-L^q$
\emph{off-diagonal estimates on balls}, denoted by $e^{-tL}\in\mathcal{O}(L^p-L^q)$,
namely,
there exist constants $\theta_1,\,\theta_2 \in[0,\fz)$ and $C,\,c\in(0,\fz)$ such that,
for any $t\in(0,\fz)$, any ball $B:=B(x_B,r_B)\subset\cx$ with $x_B\in\cx$ and $r_B\in (0,\fz)$,
and any locally $p$-integrable function $f$ on $\cx$,
\begin{align*}
\lf\{\frac{1}{\mu(B)}\int_B|e^{-tL}\lf(\chi_Bf\r)(x)|^q\,d\mu(x)\r\}^{\frac1q}\le C
\lf[\Upsilon\lf(\frac{r_B}{t^{1/2}}\r)\r]^{\theta_2}\lf\{\frac{1}{\mu(B)}\int_B
|f(x)|^p\,d\mu(x)\r\}^{\frac1p}
\end{align*}
and, for any $j\in\nn\cap [3,\,\fz)$,
\begin{align*}
&\lf\{\frac{1}{\mu(2^jB)}\int_{U_j(B)}|T_t\lf(\chi_{B}f\r)(x)|^q\,d\mu(x)\r\}^{\frac1q}\\
&\hs\le C
2^{j\theta_1}\lf[\Upsilon\lf(\frac{2^jr_B}{t^{1/2}}\r)\r]^{\theta_2}
e^{-c\frac{(2^jr_B)^{2}}{t}}\lf\{\frac{1}{\mu(B)}\int_B
|f(x)|^p\,d\mu(x)\r\}^{\frac1p}
\end{align*}
and
\begin{align*}
&\lf\{\frac{1}{\mu(B)}\int_B |T_t(\chi_{U_j(B)}f)(x)|^q\,d\mu(x)\r\}^{\frac1q}\\
&\hs\le C 2^{j\theta_1}\lf[\Upsilon\lf(\frac{2^jr_B}{t^{1/2}}\r)\r]^{\theta_2}
e^{-c\frac{(2^jr_B)^{2}}{t}}\lf\{\frac{1}{\mu(2^jB)}\int_{U_j(B)}
|f(x)|^p\,d\mu(x)\r\}^{\frac1p},
\end{align*}
where $U_j(B)$ is as in \eqref{eq ujb} and, for all $s\in(0,\fz)$,
$\Upsilon(s):=\max\lf\{s,\frac{1}{s}\r\}$. The notion of off-diagonal
estimates on balls was first introduced by Auscher and Martell \cite{am07}
in the setting of a metric space with a doubling measure,
which was operational for proving weighted estimates in \cite{am06}.
From \cite[Proposition 3.2]{am07}, we deduce that, for any $1\le p\le q\le\fz$,
$e^{-tL}\in\mathcal{O}(L^p-L^q)$ is equivalent to $e^{-tL}\in\mathcal{F}(L^p-L^q)$
in the setting of the classical Euclidean space.
By this and (ii) of this remark, we know that the reinforced off-diagonal estimates
on balls introduced in \cite{bckyy13a} are stronger than Assumption \ref{as-b}
of this article.
\end{enumerate}
\end{remark}

\begin{remark}\label{rem Assu}
Let $L$ be an operator satisfying Assumptions \ref{as-a} and \ref{as-b}.
\begin{enumerate}
\item[(i)] By Remark \ref{r-2.0}, we know that
the semigroup $e^{-zL}$ is holomorphic in $S^0_{\pi/2-\omega}$.
From this, Assumption \ref{as-b} and an argument similar to that
used in the proof of \cite[Proposition 3.1]{hlmmy11},
we deduce that, for any $k\in\zz_+$, the family $\{(tL)^ke^{-tL}\}_{t>0}$ of operators
satisfies the Davies-Gaffney estimates \eqref{eq-dg}. In particular,
for any $k\in\zz_+$ and $t\in(0,\,\fz)$,
the operator $(tL)^ke^{-tL}$ is bounded on $L^2(\rn)$.

\item[(ii)] Let $L^\ast$ be the adjoint operator of $L$ in $L^2(\rn)$.
Then, by \cite[Theorems 5.30 and 6.22 of Chapter Three]{ka95}, we know that $L^\ast$ is
also a one-to-one operator of type $\oz$ in $L^2(\rn)$.
From \cite[Lemma 2.6.2]{h06}, it follows that, for any $k\in\mathbb{Z}_+$ and $t\in(0,\,\fz)$,
$[(tL)^k e^{-tL}]^\ast=(tL^\ast)^k e^{-tL^\ast}$.
By this, (i) of this remark and an argument of duality, we find that, for any $k\in\mathbb{Z}_+$,
the family $\{(tL^\ast)^ke^{-tL^\ast}\}_{t>0}$ of operators also
satisfies \eqref{eq-dg}.

\item[(iii)] By \cite[Lemma 2.3]{hm03}, we know that there exist
positive constants $C$ and $c$ such that,
for any $t,\,s\in (0,\,\fz)$, any closed subsets $E$ and $F$ of $\rn$
and $f\in L^2(\rn)$ with $\supp f\st E$,
\begin{equation}\label{eq tt}
\lf\|e^{-sL}e^{-tL}(f)\r\|_{L^2(F)}\le C e^{-c\frac{[\dist(E,\,F)]^2}{\max\{s,\,t\}}}
\|f\|_{L^2(E)}.
\end{equation}

\item[(iv)] We point out that the assumption that $L$ is one-to-one is necessary for the
bounded holomorphic functional calculus on $L^2(\rn)$ (see \cite{m86,adm96}).
By \cite[Theorem 2.3]{cdmy96}, we further know
that, if $T$ is a one-to-one operator of type $\oz$
in $L^2(\rn)$, then $T$ has dense domain and dense range.
\end{enumerate}
\end{remark}

\begin{remark}\label{rem Assu*}
Examples of operators, which satisfy Assumptions \ref{as-a} and \ref{as-b}, include:
\begin{enumerate}
\item[(i)] the second order divergence form elliptic operator
with complex bounded coefficients as in \cite{hm09,hmm11}.
Recall that a matrix $A(x):=(A_{ij}(x))_{i,j=1}^n$ of complex-valued measurable functions on $\rn$
is said to satisfy the \emph{elliptic condition} if there exist positive constants
$\lz\le\blz$ such that, for almost every $x\in\rn$ and all $\xi,\,\eta\in\cc^n$,
$$\lz|\xi|^2\le\Re\langle A(x)\xi,\,\xi\rangle\ \ \ \text{and}\ \ \
|\langle A(x)\xi,\,\eta\rangle|\le\blz|\xi||\eta|,$$
where $\langle\cdot,\,\cdot\rangle$ denotes the \emph{inner product} in $\cc^n$ and
$\Re\xi$ denotes the \emph{real part} of $\xi$.
For such a matrix $A(x)$, the associated \emph{second order divergence form elliptic operator $L$}
is defined by setting, for any $f\in D(L)$,
\begin{equation}\label{eq op}
Lf:=-\div(A\nabla f),
\end{equation}
which is interpreted in the weak sense via a sesquilinear form.
Here and hereafter, $D(L)$ denotes the domain of $L$.
It is well known that there exists a positive constant $\oz\in[0,\,\pi/2)$ such
that the operator $L$ is one-to-one of type $\oz$ in $L^2(\rn)$ and $L$
has a bounded holomorphic functional calculus in $L^2(\rn)$
(see, for example, \cite{at98,au07,hmm11}).
Hence, $L$ satisfies Assumption \ref{as-a}.
Let $k\in\zz_+$. By \cite[Proposition 5.7(a)]{am07} (see also \cite[Corollary 3.6]{au07}),
we find that there exist positive constants $p_-(L)$ and $p_+(L)$ such that,
for all $p_-(L)<p\le q<p_+(L)$,
$(tL)^ke^{-tL}\in\mathcal{F}(L^p-L^q)$,
namely, there exist positive constants $C$ and $c$ such that, for all $t\in(0,\,\fz)$,
any closed subsets $E,\,F$ of $\rn$ and $f\in L^p(\rn)$ with $\supp f\st E$,
\begin{align}\label{eq-od}
\lf\|(tL)^k e^{-tL}(f)\r\|_{L^q(F)}
\le C t^{-\frac{n}{2}(\frac1p-\frac1q)}e^{-c\frac{[\dist(E,\,F)]^2}{t}}
\lf\|f\r\|_{L^p(E)}.
\end{align}
Moreover, by \cite[Section 3.4]{au07}, we know that
\begin{equation}\label{eq bound1}
p_-(L)\in\lf[1,\,\frac{2n}{n+2}\r)\ \ \text{for}\ n\geq 3;\ \ \ \ p_-(L)=1\ \
\text{for}\ n\in\{1,\,2\}
\end{equation}
and
\begin{equation}\label{eq bound2}
\ p_+(L)\in\lf(\frac{2n}{n-2},\,\fz\r]\ \ \text{for}\ n\geq 3;\ \ \ \ p_+(L)=\fz\ \
\text{for}\ n\in\{1,\,2\}.
\end{equation}
This implies that $L$ satisfies Assumption \ref{as-b}.
The following diagram illustrates the parameters involved in the $L^p-L^q$
off-diagonal estimates satisfied by $(tL)^ke^{-tL}$.
\begin{figure}[H]
\begin{center}
\scalebox{0.55}[0.55]{\includegraphics*[72,235][481,558]{pems-16009.eps}}
\end{center}
\end{figure}
\noindent
Here, the bottom-right corner of the shaded triangle is $(1/{p_-(L)},\,1/{p_+(L)})$,
with $p_-(L)\in[1,\,2)$ and $p_+(L)\in(2,\,\fz]$,
and, for every pair $(1/p,\,1/q)$ in the shaded region, $(tL)^ke^{-tL}\in \mathcal{F}(L^p-L^q)$.

\item[(ii)] the one-to-one non-negative self-adjoint operators $L$
having the \emph{Gaussian upper bounds}, namely,
there exist positive constants $C$ and $c$ such that, for any $t\in(0,\,\fz)$ and $x,\,y\in\rn$,
\begin{equation}\label{eq-gauss}
|p_t(x,\,y)|\le\frac{C}{t^{n/2}}\exp\lf(-c\frac{|x-y|^2}{t}\r),
\end{equation}
where $p_t$ denotes the kernel of $e^{-tL}$.
Indeed, every non-negative self-adjoint operator $L$ is an
operator of type $0$ and has a bounded holomorphic functional calculus.
Thus, $L$ satisfies Assumption \ref{as-a}.
Moreover, by \eqref{eq-gauss} and \cite[Proposition 2.2]{am07}, we know that,
for any $1\le p\le q\le\fz$, $e^{-tL}\in\mathcal{O}(L^p-L^q)$.
By the fact that $e^{-tL}\in\mathcal{O}(L^p-L^q)$ is equivalent to
$e^{-tL}\in\mathcal{F}(L^p-L^q)$ (see \cite[Proposition 3.2]{am07}),
we know that $L$ satisfies Assumption \ref{as-b}.

\item[(iii)] the \emph{Schr\"{o}dinger operator} $-\Delta+V$ on $\rn$ with the non-negative potential
$V\in L^1_{\rm loc}(\rn)$ and not identically zero
(see, for example, \cite{hlmmy11,jy11,yy14,sy10}
and related references therein).
Indeed, by \cite[Chapter 8]{hlmmy11},
we know that $-\Delta+V$ is
a particular case of (ii) of this remark.
\end{enumerate}
\end{remark}

\subsection{Variable Hardy spaces $H_L^{p(\cdot)}(\rn)$}
\hskip\parindent
In this subsection, we introduce the variable Hardy space $H_L^{p(\cdot)}(\rn)$.
We begin with recalling some notation and notions on variable Lebesgue spaces.

Let $\cp(\rn)$ be the set of all the measurable functions $p(\cdot):\ \rn\to (0,\,\fz)$
satisfying
\begin{equation}\label{eq var}
p_-:=\essinf_{x\in\rn}p(x)>0\ \ \text{and}\ \ p_+:=\esup_{x\in\rn}p(x)<\fz.
\end{equation}
A function $p(\cdot)\in\cp(\rn)$ is called a \emph{variable exponent function on $\rn$}.
For any $p(\cdot)\in\cp(\rn)$ with $p_-\in(1,\,\fz)$,
we define $p'(\cdot)\in\cp(\rn)$ by
\begin{align*}
\frac{1}{p(x)}+\frac{1}{p'(x)}=1\ \ \ \text{for all}\ \ x\in\rn.
\end{align*}
The function $p'$ is called the \emph{dual variable exponent} of $p$.

For any $p(\cdot)\in\cp(\rn)$,
the \emph{variable Lebesgue space $\vp$} is defined
to be the set of all measurable functions $f$ satisfying \eqref{eq-vl},
equipped with the \emph{Luxemburg} (or known as the \emph{Luxemburg-Nakano}) {(quasi-)norm}
\begin{equation}\label{eq norm}
\|f\|_{\vp}:=\inf\lf\{\lz\in(0,\,\fz):\ \int_\rn
\lf[\frac{|f(x)|}{\lz}\r]^{p(x)}\,dx\le 1\r\}.
\end{equation}
For more properties of the variable Lebesgue spaces, we refer the
reader to \cite{cf13,dhhr11}.

\begin{remark}\label{rem 1}
Let $p(\cdot)\in\cp(\rn)$.
\begin{enumerate}
\item[(i)] For any $\lz\in\mathbb{C}$ and $f\in\vp$, $\|\lz f\|_{\vp}=|\lz|\|f\|_{\vp}$.
In particular, if $p_-\in[1,\,\fz)$, then $\vp$ is a Banach space (see \cite[Theorem 3.2.7]{dhhr11})
and, for any $f,\,g\in\vp$,
$\|f+g\|_{\vp}\le \|f\|_{\vp}+\|g\|_{\vp}.$

\item[(ii)] For any non-trivial function $f\in\vp$, it holds true that
\begin{equation*}
\int_\rn \lf[\frac{|f(x)|}{\|f\|_{\vp}}\r]^{p(x)}\,dx=1
\end{equation*}
(see, for example, \cite[Proposition 2.21]{cf13}).

\item[(iii)] By \eqref{eq norm}, it is easy to see that,
for any $s\in(0,\,\fz)$ and $f\in\vp$,
\begin{equation*}
\||f|^s\|_{\vp}=\|f\|_{L^{sp(\cdot)}(\rn)}^s
\end{equation*}
(see also \cite[Lemma 2.3]{cw14}).
\end{enumerate}
\end{remark}

Recall that a measurable function $g\in\cp(\rn)$ is said to be
\emph{globally $\log$-H\"{o}lder continuous},
denoted by $g\in C^{\log}(\rn)$, if there exist constants $C_1,\,C_2\in(0,\,\fz)$
and $g_\fz\in\mathbb{R}$ such that, for any $x,\,y\in\rn$,
\begin{equation*}
|g(x)-g(y)|\le \frac{C_1}{\log(e+1/|x-y|)}
\end{equation*}
and
\begin{equation*}
|g(x)-g_\fz|\le \frac{C_2}{\log(e+|x|)}.
\end{equation*}

Also, recall that the \emph{Hardy-Littlewood maximal operator} $\cm$ is defined by setting,
for all $f\in L^1_{\loc}(\rn)$ and $x\in\rn$,
\begin{equation}\label{2.10x}
\cm(f)(x):=\sup_{B\ni x}\frac{1}{|B|}\int_B|f(y)|\,dy,
\end{equation}
where the supremum is taken over all balls $B$ of $\rn$ containing $x$.

\begin{remark}
Let $p(\cdot)\in C^{\log}(\rn)$ and $1<p_-\le p_+<\fz$. Then there exists a positive
constant $C$ such that, for any $f\in\vp$, $\|\cm(f)\|_{\vp}\le C\|f\|_{\vp}$
(see, for example, \cite[Theorem 4.3.8]{dhhr11}).
\end{remark}

The following Fefferman-Stein vector-valued inequality of $\cm$ on $\vp$ was
proved in \cite[Corollary 2.1]{cfmp06}.
\begin{lemma}[\cite{cfmp06}]\label{lem fs}
Let $q\in(1,\,\fz)$ and $p(\cdot)\in C^{\log}(\rn)$ with $p_-\in(1,\,\fz)$.
Then there exists a positive constant $C$ such that, for any sequence $\{f_j\}_{j\in\nn}$
of measurable functions,
\begin{align*}
\lf\|\lf\{\sum_{j=1}^\fz[\cm(f_j)]^q\r\}^{\frac1q}\r\|_{\vp}
\le C\lf\|\lf(\sum_{j=1}^\fz|f_j|^q\r)^{\frac1q}\r\|_{\vp}.
\end{align*}
\end{lemma}

Assume that the operator $L$ satisfies Assumptions \ref{as-a} and \ref{as-b}. For any $k\in\nn$,
the \emph{square function $S_{L,\,k}$ associated with $L$} is defined by setting,
for any $f\in L^2(\rn)$ and $x\in\rn$,
\begin{equation}\label{eq-1*}
S_{L,\,k}(f)(x):=\lf[\iint_{\bgz(x)}\lf|(t^2L)^ke^{-t^2L}(f)(y)\r|^2\dydt\r]^{1/2},
\end{equation}
where $\bgz(x)$ is as in \eqref{eq bgz} with $\az=1$.
In particular, when $k=1$, we write $S_L$ instead of $S_{L,\,k}$.
We notice that, for any $k\in\nn$, $S_{L,\,k}$ is bounded on $L^2(\rn)$.
Indeed, by the Fubini theorem, we know that, for any $f\in L^2(\rn)$,
\begin{align}\label{eq-1}
\int_\rn|S_{L,\,k}(f)(x)|^2\,dx
&=\int_\rn\int_0^\fz\int_{|y-x|<t}\lf|(t^2L)^ke^{-t^2L}(f)(y)\r|^2\,\frac{dy\,dt}{t^{n+1}}\,dx\noz\\
&=\int_\rn\int_0^\fz\lf|(t^2L)^ke^{-t^2L}(f)(y)\r|^2\,\frac{dt}{t}\,dy
\ls\|f\|_{L^2(\rn)}^2,
\end{align}
where the last step in \eqref{eq-1} is from \cite[Theorem F]{adm96}
(see also \cite[(4.1)]{hlmmy11}).

We now introduce the variable Hardy spaces associated with the operator $L$.
\begin{definition}\label{def vhp}
Let $p(\cdot)\in \cp(\rn)$ satisfy $p_+\in(0,\,1]$
and $L$ be an operator satisfying Assumptions \ref{as-a} and \ref{as-b}.
The \emph{variable Hardy space} $H_{L}^{p(\cdot)}(\rn)$ is defined as the
completion of the space
\begin{align*}
\lf\{f\in L^2(\rn):\ \|S_L(f)\|_{\vp}<\fz\r\}
\end{align*}
with respect to the \emph{quasi-norm}
\begin{align*}
\|f\|_{H_L^{p(\cdot)}(\rn)}:=\|S_L(f)\|_{\vp}=\inf\lf\{\lz\in(0,\,\fz):\
\int_\rn \lf[\frac{S_L(f)(x)}{\lz}\r]^{p(x)}\,dx\le 1\r\}.
\end{align*}
\end{definition}

\begin{remark}\label{rem vhp}
\begin{enumerate}
\item[(i)] In particular, when $p(\cdot)\equiv $ constant $\in(0,\,1]$,
$\vhp$ was studied in \cite{dl13,bl11} as a special case.
We refer the reader to \cite{hlmmy11,jy11,yy14}
for more progresses on Hardy spaces associated with operators satisfying
the Davies-Gaffney estimates.

\item[(ii)] If $L$ is a non-negative self-adjoint operator having the Gaussian upper bounds
as in Remark \ref{rem Assu*}(ii), the variable Hardy space $\vhp$ was studied in \cite{yzo15, zy15}.
Moreover, when $L:=-\Delta$ is the Laplace operator on $\rn$,
by \cite[Theorem 5.3]{yzo15}, we conclude that, if $p(\cdot)\in C^{\log}(\rn)$
with $\frac{n}{n+1}<p_-\le p_+\le 1$ and $\frac2{p_-}-\frac1{p_+}<\frac{n+1}{n}$,
then $H^{p(\cdot)}_L(\rn)$ and $H^{p(\cdot)}(\rn)$ coincide with equivalent quasi-norms,
where $H^{p(\cdot)}(\rn)$ stands for the \emph{variable Hardy space}
(see also Definition \ref{def hardy} below).

\item[(iii)] Let $\vz:\ \rn\times[0,\fz)\to[0,\fz)$ be a growth function as in \cite{bckyy13a}
and $L$ be an operator satisfying the reinforced off-diagonal estimates on balls
as in Remark \ref{rem-1x}(iii).
Then Bui et al. \cite{bckyy13a} introduced the Musielak-Orlicz-Hardy space $H_{\varphi,L}(\rn)$
associated with the operator $L$ via the Lusin area function.
Recall that the Musielak-Orlicz space $L^{\vz}(\rn)$ is defined to be the set
of all measurable functions $f$ on $\rn$ such that
$$\|f\|_{L^{\vz}(\rn)}:=\inf\lf\{\lz\in(0,\fz):\
\int_\rn\vz(x,|f(x)|/\lz)\,dx\le1\r\}<\fz.$$
Observe that, if
\begin{equation}\label{vz}
\vz(x,t):=t^{p(x)}\quad {\rm for\ all}\quad x\in\rn\quad{\rm and}\quad t\in[0,\fz),
\end{equation}
then $L^{\vz}(\rn)=L^{p(\cdot)}(\rn)$. However, a general
Musielak-Orlicz function $\vz$ satisfying all the assumptions
in \cite{bckyy13a} may not have the form
as in \eqref{vz} (see \cite{ky14}). On the other hand, it was proved in
\cite[Remark 2.23(iii)]{yyz14} that there exists a variable exponent function
$p(\cdot)\in C^{\log}(\rn)$,
but $t^{p(\cdot)}$ is not a uniformly Muckenhoupt weight
which was required in \cite{bckyy13a}.
Thus, Musielak-Orlicz-Hardy spaces associated with operators in \cite{bckyy13a}
and variable Hardy spaces associated with operators in this article
do not cover each other.
\end{enumerate}
\end{remark}

\section{Molecular characterization of $H_L^{p(\cdot)}(\rn)$\label{s3}}
\hskip\parindent
In this section, we first recall some properties of variable tent spaces
from \cite{yzo15,zyl14}. Then we establish the molecular characterization of
$\vhp$.

\subsection{Variable tent spaces}
\hskip\parindent
For any measurable function $f$ on $\rnn$
and $x\in\rn$, define
\begin{equation*}
A(f)(x):=\lf[\iint_{\bgz(x)}|f(y,\,t)|^2\dydt\r]^{1/2},
\end{equation*}
where $\bgz(x)$ is as in \eqref{eq bgz} with $\az=1$.
For any $q\in(0,\,\fz)$, the \emph{tent space $T^q(\rnn)$} is defined to be
the space of all measurable functions $f$ such that
$\|f\|_{T^q(\rnn)}:=\|A(f)\|_{L^q(\rn)}<\fz$,
which was first introduced by Coifmann et al. in \cite{cms85}.

The following lemma is just \cite[Theorem 2]{cms85}.
\begin{lemma}[\cite{cms85}]\label{lem 0}
Let $p\in(1,\,\fz)$. Then, for any $f\in T^p(\rnn)$ and $g\in T^{p'}(\rnn)$,
the pairing
\begin{align*}
\langle f,\,g\rangle:=\iint_{\rnn} f(x,\,t)\ov{g(x,\,t)}\,\frac{dx\,dt}{t}
\end{align*}
realizes $T^{p'}(\rnn)$ as the dual of $T^p(\rnn)$, up to equivalent norms,
where $1/p+1/p'=1$.
\end{lemma}

\begin{definition}\label{def tent}
Let $p(\cdot)\in\cp(\rn)$. The \emph{variable tent space $T^{p(\cdot)}(\rnn)$}
is defined to be the space of all measurable functions $f$ such that
\begin{equation*}
\|f\|_{T^{p(\cdot)}(\rnn)}:=\|A(f)\|_{\vp}<\fz.
\end{equation*}
\end{definition}

For any open set $O\st\rn$, the \emph{tent} over $O$  is defined by setting
\begin{align*}
\wh{O}:=\lf\{(y,\,t)\in\rnn:\ \dist \lf(y,\,O^\com\r)\geq t\r\}.
\end{align*}

Let $p(\cdot)\in\cp(\rn)$. Recall that a measurable
function $a$ on $\rnn$ is called a \emph{$(p(\cdot),\,\fz)$-atom} if there
exists a ball $B\st\rn$ such that
\begin{enumerate}
\item[(i)] $\supp a\st \wh{B}$;

\item[(ii)] for all $q\in(1,\,\fz)$, $\|a\|_{T^q(\rnn)}\le |B|^{1/q}\|\chi_B\|_{\vp}^{-1}$.
\end{enumerate}

We point out that the $(p(\cdot),\,\fz)$-atom was first introduced in \cite{zyl14}.
For any $p(\cdot)\in\cp(\rn)$ with $0<p_-\le p_+\le 1$,
any sequences $\{\lz_j\}_{j\in\nn}\st\cc$ and $\{B_j\}_{j\in\nn}$ of balls in $\rn$, let
\begin{equation*}
\ca(\{\lz_j\}_{j\in\nn},\,\{B_j\}_{j\in\nn}):=\lf\|\lf\{\sum_{j\in\nn}
\lf[\frac{|\lz_j|\chi_{B_j}}{\|\chi_{B_j}\|_{\vp}}\r]^{p_-}\r\}^{\frac1{p_-}}\r\|_{\vp}.
\end{equation*}

The following lemma establishes the atomic decomposition of $T^{p(\cdot)}(\rn)$,
which is a slight variant of \cite[Theorem 2.16]{zyl14} (see also \cite[Lemma 3.3]{yzo15}).

\begin{lemma}\label{lem 2}
Let $p(\cdot)\in C^{\log}(\rn)$ with $p_+\in(0,\,1]$. Then, for any $f\in T^{p(\cdot)}(\rn)$,
there exist $\{\lz_j\}_{j\in\nn}\st\cc$ and a family $\{a_j\}_{j\in\nn}$ of $(p(\cdot),\,\fz)$-atoms
such that, for almost every $(x,\,t)\in\rnn$,
\begin{equation}\label{eq decom}
f(x,\,t)=\sum_{j\in\nn}\lz_j a_j(x,\,t)
\end{equation}
and
$$\ca(\{\lz_j\}_{j\in\nn},\,\{B_j\}_{j\in\nn})
\le C\|f\|_{T^{p(\cdot)}(\rnn)},$$
where, for any $j\in\nn$, $B_j$ is the ball associated with $a_j$
and $C$ a positive constant independent of $f$.
\end{lemma}
The proof of Lemma \ref{lem 2} is a slight modification of the proof of \cite[Theorem 2.16]{zyl14}
via replacing the cubes therein by balls of $\rn$, the details being omitted.

\begin{remark}\label{rem 2}
\begin{enumerate}
\item[(i)] Let $p(\cdot)\in C^{\log}(\rn)$ with $p_+\in (0,\,1]$ and $q\in(0,\,\fz)$.
By \cite[Corollary 3.4 and Remark 3.6]{yzo15}, we know that,
if $f\in T^{p(\cdot)}(\rnn)\cap T^q(\rnn)$, then
\eqref{eq decom} holds true in both $T^{p(\cdot)}(\rnn)$ and $T^q(\rnn)$.

\item[(ii)] Let $p(\cdot)\in C^{\log}(\rn)$ with $p_+\in (0,\,1]$. Then, by \cite[Remark 4.4]{ns12},
we know that, for any $\{\lz_j\}_{j\in\nn}\st\cc$ and $\{B_j\}_{j\in\nn}$ of balls in $\rn$,
$\sum_{j\in\nn}|\lz_j|\le \ca(\{\lz_j\}_{j\in\nn},\,\{B_j\}_{j\in\nn}).$
\end{enumerate}
\end{remark}

\subsection{Molecular characterization of $H_L^{p(\cdot)}(\rn)$}
\hskip\parindent
In this subsection, we establish the molecular characterization of $\vhp$.
We begin with introducing some notions.

\begin{definition}\label{def mol}
Let $L$ satisfy Assumptions \ref{as-a} and \ref{as-b}
and $p(\cdot)\in \cp(\rn)$ with $p_+\in(0,\,1]$.
Assume $M\in\nn$ and $\varepsilon\in(0,\,\fz)$.
A function $m\in L^2(\rn)$ is called a \emph{$(p(\cdot),\,M,\,\varepsilon)_L$-molecule}
if $m\in R(L^M)$ (the range of $L^M$) and there exists a ball $B:= B(x_B,\,r_B)\st\rn$
with $x_B\in\rn$ and $r_B\in(0,\,\fz)$ such that, for any $k\in\{0,\,\ldots,\,M\}$
and $j\in\zz_+$,
\begin{equation*}
\lf\|(r_B^{-2}L^{-1})^k (m)\r\|_{L^2(\ujb)}\le
2^{-j\varepsilon}|2^j B|^{1/2}\|\chi_B\|^{-1}_{\vp},
\end{equation*}
where, for any $j\in\zz_+$, $U_j(B)$ is as in \eqref{eq ujb}.
\end{definition}

\begin{remark}\label{r-0902}
Let $m$ be a $(p(\cdot),\,M,\,\vez)_L$-molecule as in Definition \ref{def mol}
associated with the ball $B\st \rn$. If $\vez\in(\frac n2,\fz)$, then it is easy to see that,
for any $k\in\{0,\dots,M\}$,
\begin{align*}
\lf\|(r_B^{-2}L^{-1})^k(m)\r\|_{L^2(\rn)}\le C |B|^{1/2}\|\chi_B\|_{L^{p(\cdot)}(\rn)}^{-1}
\end{align*}
with $C$ being a positive constant independent of $m$, $k$ and $B$.
\end{remark}

\begin{definition}\label{def m-Hardy}
Let $L$ satisfy Assumptions \ref{as-a} and \ref{as-b}
and $p(\cdot)\in \cp(\rn)$ with $p_+\in(0,\,1]$.
Assume $M\in\nn$ and $\varepsilon\in(0,\,\fz)$.
For a measurable function $f$ on $\rn $,
$f=\sum_{j=1}^\fz \lz_j m_j$ is called a
\emph{molecular $(p(\cdot),\,M,\,\varepsilon)$-representation} of $f$
if $\{m_j\}_{j\in\nn}$ is a family of $\size$-molecules,
the summation converges in $L^2(\rn)$ and $\{\lz_j\}_{j\in\nn}\st\cc$ satisfies that
$\ca(\{\lz_j\}_{j\in\nn},\,\{B_j\}_{j\in\nn})<\fz,$
where, for any $j\in\nn$, $B_j$ is the ball associated with $m_j$.
Let
\begin{align*}
\mol:=\{f:\ f\ \text{has a molecular}\
(p(\cdot),\,M,\,\varepsilon)\text{-representation}\}.
\end{align*}
Then the \emph{variable molecular Hardy space $\mhp$} is defined as
the completion of $\mol$ with respect to the
\emph{quasi-norm}
\begin{align*}
\|f\|_{H_{L,\,M}^{p(\cdot),\,\vez}(\rn)}
&:=\inf\Bigg\{\ca\lf(\{\lz_j\}_{j\in\nn},\,\{B_j\}_{j\in\nn}\r):\ \\
&\qquad\qquad f=\sum_{j=1}^\fz\lz_j m_j\ \text{is a molecular}\
(p(\cdot),\,M,\,\varepsilon)\text{-representation}\Bigg\},
\end{align*}
where the infimum is taken over all the molecular
$(p(\cdot),\,M,\,\varepsilon)$-representations of $f$ as above.
\end{definition}

To establish the molecular characterization of $\vhp$, we need the following
two technical lemmas.
Lemma \ref{lem key} is a slight variant of \cite[Lemma 4.1]{s13} with cubes therein replaced by balls here.
\begin{lemma}\label{lem key}
Let $p(\cdot)\in C^{\log}(\rn)$, $\underline{p}:=\min\{p_-,\,1\}$ and $q\in[1,\,\fz)$.
Then there exists a positive constant $C$ such that, for any sequence $\{B_j\}_{j\in\nn}$
of balls in $\rn$, $\{\lz_j\}_{j\in\nn}\st\mathbb{C}$ and functions $\{a_j\}_{j\in\nn}$ satisfying
that, for any $j\in\nn$, $\supp a_j\st B_j$ and $\|a_j\|_{L^q(\rn)}\le |B_j|^{1/q}$,
\begin{align}\label{eq key}
\lf\|\lf[\sum_{j=1}^\fz|\lz_j a_j|^{\uz{p}}\r]^{\frac{1}{\uz{p}}}\r\|_{\vp}
\le C\lf\|\lf[\sum_{j=1}^\fz|\lz_j \chi_{B_j}|^{\uz{p}}\r]^{\frac{1}{\uz{p}}}\r\|_{\vp}.
\end{align}
\end{lemma}

\begin{proof}
For any sequence $\{B_j\}_{j\in\nn}$ of balls in $\rn$, we can find a sequence $\{Q_j\}_{j\in\nn}$
of cubes in $\rn$ such that, for any $j\in\nn$,
\begin{align}\label{eq key1}
B_j\st Q_j\st \sqrt{n}B_j.
\end{align}
It is easy to see that, for any $x\in\sqrt{n}B_j$,
$\cm(\chi_{B_j})(x)\geq \frac{|B_j|}{|\sqrt{n}B_j|}=n^{-\frac{n}2}.$
Hence, for any $j\in\nn$ and $x\in\rn$,
\begin{align}\label{eq key2}
\chi_{Q_j}(x)\le\chi_{\sqrt{n}B_j}(x)\ls \cm(\chi_{B_j})(x).
\end{align}
By \cite[Lemma 4.1]{s13}, we know that \eqref{eq key} holds true with
$B_j$ replaced by $Q_j$. From this, \eqref{eq key1}, \eqref{eq key2},
Remark \ref{rem 1}(iii) and Lemma \ref{lem fs},
we deduce that
\begin{align*}
\lf\|\lf[\sum_{j=1}^\fz|\lz_j a_j|^{\uz{p}}\r]^{\frac{1}{\uz{p}}}\r\|_{\vp}
&\ls \lf\|\lf[\sum_{j=1}^\fz|\lz_j \chi_{Q_j}|^{\uz{p}}\r]^{\frac{1}{\uz{p}}}\r\|_{\vp}\\
&\ls\lf\|\lf\{\sum_{j=1}^\fz|\lz_j|^{\uz{p}}[\cm(\chi_{B_j})]^2\r\}^{\frac12}
\r\|_{L^{\frac{2p(\cdot)}{\uz{p}}}(\rn)}^{\frac2{\uz{p}}}\\
&\ls\lf\|\lf\{\sum_{j=1}^\fz|\lz_j|^{\uz{p}}\chi_{B_j}\r\}^{\frac12}
\r\|_{L^{\frac{2p(\cdot)}{\uz{p}}}(\rn)}^{\frac2{\uz{p}}}\\
&\sim\lf\|\lf[\sum_{j=1}^\fz|\lz_j \chi_{B_j}|^{\uz{p}}\r]^{\frac{1}{\uz{p}}}\r\|_{\vp}.
\end{align*}
This finishes the proof of Lemma \ref{lem key}.
\end{proof}

The following lemma is just \cite[Lemma 2.6]{zyl14} with cubes therein replaced by balls
here (see also \cite[Lemma 3.13]{yzo15} and \cite[Corollary 3.4]{izu10}).
\begin{lemma}\label{lem 3}
Let $p(\cdot)\in C^{\log}(\rn)$.  Then there exists a positive constant $C$ such that,
for any balls $B_1$ and $B_2$ of $\rn$ satisfying $B_1\st B_2$,
\begin{align}\label{eq 1.1}
C^{-1}\lf(\frac{|B_1|}{|B_2|}\r)^{\frac{1}{p_-}}
\le\frac{\|\chi_{B_1}\|_{\vp}}{\|\chi_{B_2}\|_{\vp}}
\le C\lf(\frac{|B_1|}{|B_2|}\r)^{\frac{1}{p_+}}.
\end{align}
\end{lemma}

\begin{proof}
For $i\in\{1,\,2\}$, let $B_i:=B(x_i,\,r_i)$ with $x_i\in\rn$ and
$r_i\in(0,\,\fz)$.
For any $x\in\rn$ and $r\in(0,\,\fz)$, denote by $Q(x,\,r)$ the open cube
centered at $x$ with the side length $r$.
Let $Q_1:=Q(x_1,\,\frac{2r_1}{\sqrt{n}})$ and $Q_2:=Q(x_2,\,2r_2)$.
It is easy to see that $Q_1\st B_1\st B_2\st Q_2$,
\begin{align}\label{eq 1.4}
|B_1|\sim |Q_1|\ \ \ \text{and}\ \ \ |B_2|\sim |Q_2|.
\end{align}
Let $\wz{Q}_1:=Q(x_1,\,2r_1)$. Then we have $Q_1\st B_1\st \wz{Q}_1$.
Hence, we obtain
$$\|\chi_{Q_1}\|_{\vp}\le\|\chi_{B_1}\|_{\vp}\le\lf\|\chi_{\wz{Q}_1}\r\|_{\vp}.$$
By \cite[Lemma 2.6]{zyl14}, we know that $\|\chi_{Q_1}\|_{\vp}\sim\|\chi_{\wz{Q}_1}\|_{\vp}$.
Thus, we have
\begin{align}\label{eq 1.2}
\|\chi_{B_1}\|_{\vp}\sim\|\chi_{Q_1}\|_{\vp}.
\end{align}
Similarly, we obtain
\begin{align}\label{eq 1.3}
\|\chi_{B_2}\|_{\vp}\sim\|\chi_{Q_2}\|_{\vp}.
\end{align}
On the other hand, by the fact that $Q_1\st Q_2$ and \cite[Lemma 2.6]{zyl14}, we find that
\begin{align*}
\lf(\frac{|Q_1|}{|Q_2|}\r)^{\frac{1}{p_-}}
\ls\frac{\|\chi_{Q_1}\|_{\vp}}{\|\chi_{Q_2}\|_{\vp}}
\ls \lf(\frac{|Q_1|}{|Q_2|}\r)^{\frac{1}{p_+}},
\end{align*}
which, combined with \eqref{eq 1.2}, \eqref{eq 1.3} and \eqref{eq 1.4},
implies \eqref{eq 1.1}. This finishes the proof of Lemma \ref{lem 3}.
\end{proof}

We now turn to establish the molecular characterization of $\vhp$.
\begin{proposition}\label{pro 1}
Let $L$ satisfy Assumptions \ref{as-a} and \ref{as-b}
and $p(\cdot)\in C^{\log}(\rn)$ with $p_+\in(0,\,1]$.
Assume $M\in\nn\cap(\frac n2[\frac1{p_-}-\frac12],\fz)$
and $\varepsilon\in(\frac{n}{p_-},\,\fz)$.
Then there exists a positive constant $C$ such that,
for any $f\in\mol$, $\|f\|_{\vhp}\le C\|f\|_{\mhp}.$
\end{proposition}

\begin{proof}
Let $f\in\mol$. Then,
by Definition \ref{def m-Hardy}, we know that there exist $\{\lz_j\}_{j\in\nn}\st\mathbb{C}$
and a family $\{m_j\}_{j\in\nn}$ of $\size$-molecules,
associated with balls $\{B_j\}_{j\in\nn}$ of $\rn$, such that
\begin{equation}\label{eq 3.0}
f=\sum_{j=1}^\fz \lz_j m_j\ \ \text{in}\ \ L^2(\rn)
\end{equation}
and
\begin{equation}\label{eq 3.0x}
\|f\|_{\mhp}\sim\ca(\{\lz_j\}_{j\in\nn},\,\{B_j\}_{j\in\nn}).
\end{equation}
By \eqref{eq 3.0} and the fact that $S_L$ is bounded on $L^2(\rn)$
(see \eqref{eq-1}), we find that
\begin{align*}
\lim_{N\to\fz}\lf\|S_L(f)-S_L\lf(\sum_{j=1}^N\lz_j m_j\r)\r\|_{L^2(\rn)}=0,
\end{align*}
which implies that there exists a subsequence of $\{S_L(\sum_{j=1}^N\lz_j m_j)\}_{N\in\nn}$
(without loss of generality, we may use the same notation as the original sequence)
such that, for almost every $x\in\rn$,
\begin{align*}
S_L(f)(x)=\lim_{N\to\fz}S_L\lf(\sum_{j=1}^N\lz_j m_j\r)(x).
\end{align*}
Hence, for almost every $x\in\rn$, it holds true that
\begin{align*}
S_L(f)(x)\le \sum_{j=1}^\fz|\lz_j|S_L(m_j)(x)
=\sum_{j=1}^\fz\sum_{i=0}^\fz|\lz_j|S_L(m_j)(x)\chi_{U_i(B_j)}(x),
\end{align*}
where, for each $j\in\nn$ and $i\in\zz_+$, $U_i(B_j)$ is defined as in \eqref{eq ujb} with
$B$ replaced by $B_j$.
From this, Remark \ref{rem 1}(iii) and the fact that $p_-\in(0,\,1]$, it follows that
\begin{align}\label{eq 3.2}
\|S_L(f)\|_{\vp}^{p_-}
&=\lf\|[S_L(f)]^{p_-}\r\|_{L^{\frac{p(\cdot)}{p_-}}(\rn)}\noz\\
&\le \sum_{i=0}^\fz\lf\|\sum_{j=1}^\fz|\lz_j|^{p_-}
[S_L(m_j)\chi_{U_i(B_j)}]^{p_-}\r\|_{L^{\frac{p(\cdot)}{p_-}}(\rn)}\noz\\
&= \sum_{i=0}^\fz\lf\|\lf\{\sum_{j=1}^\fz|\lz_j|^{p_-}
[S_L(m_j)\chi_{U_i(B_j)}]^{p_-}\r\}^{\frac{1}{p_-}}\r\|^{p_-}_{\vp}.
\end{align}

To prove Proposition \ref{pro 1}, it suffices to show that
there exist positive constants $C$ and $\tz\in(\frac{n}{p_-},\,\fz)$ such that,
for any $i\in\zz_+$ and $\size$-molecule $m$, associated with ball $B:= B(x_B,\,r_B)$
with $x_B\in\rn$ and $r_B\in(0,\,\fz)$,
\begin{equation}\label{eq 3.1}
\|S_L(m)\|_{L^2(U_i(B))}\le C 2^{-i\tz}|2^iB|^{1/2}\lf\|\chi_B\r\|_{\vp}^{-1}.
\end{equation}
Indeed, by \eqref{eq 3.1}, we find that, for any $i\in\zz_+$ and $j\in\nn$,
\begin{align}\label{eq 3.3}
\lf\|2^{i\tz}\lf\|\chi_{B_j}\r\|_{\vp}S_L(m_j)\chi_{U_i(B_j)}\r\|_{L^2(\rn)}\ls |2^iB_j|^{\frac12}.
\end{align}
Notice that, for any $x\in\rn$,
\begin{align}\label{eq 3.3x}
\chi_{2^iB_j}(x)\le 2^{in}\cm(\chi_{B_j})(x),
\end{align}
where $\cm$ is the Hardy-Littlewood maximal function defined in \eqref{2.10x}.
Since $\tz\in(\frac{n}{p_-},\,\fz)$,
we can choose a positive constant $r\in(0,\,p_-)$ such that $\tz\in(\frac{n}r,\fz)$.
By this, \eqref{eq 3.3}, Lemmas \ref{lem key} and \ref{lem fs}, \eqref{eq 3.3x}  and
Remark \ref{rem 1}(iii), we conclude that
\begin{align*}
&\lf\|\lf\{\sum_{j=1}^\fz \lf[|\lz_j|S_L(m_j)\chi_{U_i(B_j)}\r]^{p_-}\r\}^{\frac{1}{p_-}}\r\|_{\vp}\\
&\hs\ls\lf\|\lf\{\sum_{j=1}^\fz
\lf[2^{-i\tz}\|\chi_{B_j}\|_{\vp}^{-1}|\lz_j|\chi_{2^iB_j}\r]^{p_-}\r\}^
{\frac{1}{p_-}}\r\|_{\vp}\\
&\hs\ls 2^{-i(\tz-\frac{n}{r})}\lf\|\lf\{\sum_{j\in\nn}
\lf[\cm\lf(\frac{|\lz_j|^r}{\|\chi_{B_j}\|_{\vp}^r}\chi_{B_j}\r)\r]^{\frac{p_-}{r}}\r\}^{\frac1{p_-}}
\r\|_{\vp}\\
&\hs\sim 2^{-i(\tz-\frac{n}{r})}\lf\|\lf\{\sum_{j\in\nn}
\lf[\cm\lf(\frac{|\lz_j|^r}{\|\chi_{B_j}\|_{\vp}^r}\chi_{B_j}\r)\r]^{\frac{p_-}{r}}\r\}^{\frac{r}{p_-}}
\r\|^{\frac1r}_{L^{\frac{p(\cdot)}{r}}(\rn)}\\
&\hs\ls 2^{-i(\tz-\frac{n}{r})}\lf\|\lf\{\sum_{j\in\nn}
\lf[\frac{|\lz_j|^r}{\|\chi_{B_j}\|_{\vp}^r}\chi_{B_j}\r]^{\frac{p_-}{r}}\r\}^{\frac{r}{p_-}}
\r\|^{\frac1r}_{L^{\frac{p(\cdot)}{r}}(\rn)}\\
&\hs\sim 2^{-i(\tz-\frac{n}{r})}\mathcal{A}(\{\lz_j\}_{j\in\nn},\,\{B_j\}_{j\in\nn}).
\end{align*}
From this, \eqref{eq 3.2}, \eqref{eq 3.0x} and the fact that
$\tz\in(\frac{n}r,\,\fz)$, we deduce that,
for any $f\in\mol$,
\begin{align*}
\|f\|_{\vhp}
&=\|S_L(f)\|_{\vp}
\ls\lf\{\sum_{i=0}^\fz 2^{-i(\tz-\frac{n}{r})}\r\}^{\frac{1}{p_-}}
\|f\|_{\mhp}
\sim\|f\|_{\mhp},
\end{align*}
which is the desired result.

Next, we prove \eqref{eq 3.1}. Indeed,
when $i\in\{0,\,\ldots,\,10\}$, since $S_L$ is bounded on $L^2(\rn)$
(see \eqref{eq-1}), by the definition of $(p(\cdot),M,\vez)_L$-molecules, we have
\begin{equation}\label{eq 3.y}
\|S_L(m)\|_{L^2(U_i(B))}\ls\|m\|_{L^2(\rn)}\ls|B|^{\frac12}\|\chi_B\|_{\vp}^{-1}.
\end{equation}
When $i\in\zz_+\cap[11,\,\fz)$, for any given $\eta\in(0,\,1)$, we write
\begin{align}\label{eq 3.7y}
\|S_L(m)\|_{L^2(U_i(B))}
&=\lf[\int_{U_i(B)}\int_0^\fz
\int_{B(x,\,t)}|t^2Le^{-t^2L}(m)(y)|^2\dydt\,dx\r]^{\frac12}\noz\\
&\le\lf[\int_{U_i(B)}\int_0^{2^{i\eta}r_B}
\int_{B(x,\,t)}|t^2Le^{-t^2L}(m)(y)|^2\dydt\,dx\r]^{\frac12}\noz\\
&\hs+\lf[\int_{U_i(B)}\int_{2^{i\eta}r_B}^\fz
\int_{B(x,\,t)}\cdots\,dx\r]^{\frac12}
=:{\rm I}+{\rm II}.
\end{align}

To estimate ${\rm II}$, by Remark \ref{rem Assu}(i), we find that,
for any $k\in\zz_+$ and $t\in(0,\,\fz)$, $(tL)^ke^{-tL}$ is bounded on $L^2(\rn)$.
From this, it follows that
\begin{align}\label{eq 3.4}
{\rm II}
&\le\lf[\int_{U_i(B)}\int_{2^{i\eta}r_B}^\fz
\int_\rn\lf|(t^2L)^{M+1}e^{-t^2L}\lf(L^{-M}(m)\r)(y)\r|^2\,dy\,\frac{dt}{t^{4M+n+1}}
\,dx\r]^{\frac12}\noz\\
&\ls\lf[\int_{U_i(B)}\int_{2^{i\eta}r_B}^\fz
\lf\|L^{-M}(m)\r\|_{L^2(\rn)}^2\,\frac{dt}{t^{4M+n+1}}\,dx\r]^{\frac12}\noz\\
&\ls\lf\|L^{-M}(m)\r\|_{L^2(\rn)}\lf[\int_{U_i(B)}\lf(\frac{1}{2^{i\eta}r_B}\r)
^{4M+n}\,dx\r]^{\frac12}\noz\\
&\ls 2^{-i\eta(2M+\frac{n}{2})}2^{\frac{in}{2}}\lf\|L^{-M}(m)\r\|_{L^2(\rn)}
r_B^{-2M}.
\end{align}
By the fact that $m$ is a $\size$-molecule, $\vez>\frac{n}{p_-}>\frac{n}2$, Remark \ref{r-0902}
and \eqref{eq 3.4}, we know that
\begin{align}\label{eq 3.7z}
{\rm II}\ls 2^{-i\eta(2M+\frac{n}2)}|2^iB|^{\frac12}\|\chi_B\|_{\vp}^{-1}.
\end{align}

To estimate ${\rm I}$, for any $i\in\zz_+\cap[11,\,\fz)$, let
\begin{align*}
S_i(B):=\lf(2^{i+1}B\r)\setminus \lf(2^{i-2}B\r)\ \ \ \
\text{and}\ \ \ \ \wz{S}_i(B):=\lf(2^{i+2}B\r)\setminus \lf(2^{i-3}B\r).
\end{align*}
If $t\in(0,\,2^{i\eta}r_B)$ and $x\in U_i(B)$, then it is easy to see that
$B(x,\,t)\st S_i(B)$.
From this, we deduce that
\begin{align}\label{eq 3.7x}
{\rm I}
&\le\lf[\int_{U_i(B)}\int_0^{2^{i\eta}r_B}\int_{S_i(B)}
\lf|t^2Le^{-t^2L}\lf(m\chi_{\lf[\wz{S}_i(B)\r]^\com}\r)(y)\r|^2\dydt\,dx\r]^{\frac12}\noz\\
&\hs+\lf[\int_{U_i(B)}\int_0^{2^{i\eta}r_B}\int_{B(x,\,t)}\lf|t^2Le^{-t^2L}
\lf(m\chi_{\wz{S}_i(B)}\r)(y)\r|^2\dydt\,dx\r]^{\frac12}\noz\\
&=:{\rm I}_1+{\rm I}_2.
\end{align}
For ${\rm I}_2$, notice that, for $i\in\zz_+\cap[11,\fz)$,
$$\wz S_i(B)\st \bigcup_{k=-2}^2 U_{i+k}(B),$$
then, by the boundedness of $S_L$ on $L^2(\rn)$ (see \eqref{eq-1}),
we obtain
\begin{align}\label{eq 3.7}
{\rm I}_2\le\lf\|S_L\lf(m\chi_{\wz{S}_i(B)}\r)\r\|_{L^2(\rn)}
\ls\lf\|m\chi_{\wz{S}_i(B)}\r\|_{L^2(\rn)}
\ls 2^{-i\vez}|2^iB|^{\frac12}\|\chi_B\|_{\vp}^{-1}.
\end{align}
For ${\rm I}_1$, by Remark \ref{rem Assu}(i),
we know that $\{tLe^{-tL}\}_{t>0}$ satisfies the Davies-Gaffney estimates
\eqref{eq-dg}. From this and the fact that
$\dist([\wz{S}_i(B)]^\com,\,S_i(B))\sim 2^ir_B,$
it follows that
\begin{align}\label{eq 3.6}
{\rm I}_1
&\ls\lf[\int_{U_i(B)}\int_0^{2^{i\eta}r_B}e^{-c\frac{(2^ir_B)^2}{t^2}}\|m\|^2_{L^2(\rn)}
\,\frac{dt}{t^{n+1}}\,dx\r]^{\frac12}\noz\\
&\ls\|m\|_{L^2(\rn)}\lf[\int_{U_i(B)}\int_0^{2^{i\eta}r_B}\lf(\frac{t}{2^ir_B}\r)^N\,
\frac{dt}{t^{n+1}}\,dx\r]^{\frac12}\noz\\
&\ls\|m\|_{L^2(\rn)}|2^iB|^{\frac12}2^{-\frac{i}{2}[N(1-\eta)+\eta n]}|B|^{-\frac12},
\end{align}
where $c$ is as in \eqref{eq tt} and $N\in(n+1,\,\fz)$ is determined later.
This, together with Remark \ref{r-0902} and \eqref{eq 3.6}, implies that
\begin{align*}
{\rm I}_1\ls 2^{-\frac{i}{2}[N(1-\eta)+\eta n]}|2^iB|^{\frac12}\|\chi_B\|_{\vp}^{-1}.
\end{align*}
Combining this, \eqref{eq 3.7}, \eqref{eq 3.7x}, \eqref{eq 3.7y} and
\eqref{eq 3.7z}, we find that, for any $\size$-molecule $m$
associated with ball $B\st\rn$ and $i\in\zz_+\cap[11,\,\fz)$,
\begin{align}\label{eq 3.y1}
\|S_L(m)\|_{L^2(U_i(B))}\ls 2^{-i\tz}|2^iB|^{1/2}\lf\|\chi_B\r\|_{\vp}^{-1},
\end{align}
where
\begin{align*}
\tz:=\min\lf\{\frac{1}{2}[N(1-\eta)+\eta n],\,\vez,\,\eta\lf(2M+\frac{n}{2}\r)\r\}.
\end{align*}
By the fact that $M\in(\frac n2[\frac1{p_-}-\frac12],\fz)$, we
can choose some $\eta\in(0,\,1)$ such that
$\eta(2M+\frac{n}{2})>\frac{n}{p_-}$.
Then, by taking $N:=\frac{2n}{(1-\eta)p_-}$ and
$\vez\in(\frac{n}{p_-},\,\fz)$,
we find that $\tz\in(\frac{n}{p_-},\,\fz)$,
which, together with \eqref{eq 3.y1} and \eqref{eq 3.y}, implies \eqref{eq 3.1}.
This finishes the proof of Proposition \ref{pro 1}.
\end{proof}

Let $M\in\nn$ and $L$ satisfy Assumptions \ref{as-a} and \ref{as-b}.
For any $F\in T^2(\rnn)$ and $x\in\rn$, define
\begin{equation}\label{eq pi}
\pi_{M,\,L}(F)(x):=\int_0^\fz(t^2L)^{M+1}e^{-t^2L}(F(\cdot,\,t))(x)\,\frac{dt}{t}.
\end{equation}

\begin{lemma}\label{lem 5}
Let $L$ satisfy Assumptions \ref{as-a} and \ref{as-b}
and $p(\cdot)\in \cp(\rn)$ with $p_+\in(0,\,1]$.
Assume that $A$ is a $(p(\cdot),\,\fz)$-atom associated with ball $B\st\rn$.
Then, for any $M\in\nn$ and $\vez\in(0,\,\fz)$, there exists a
positive constant $C_{(M,\,\vez)}$, depending on $M$ and $\vez$,
such that $C_{(M,\,\vez)}\pi_{M,\,L}(A)$
is a $\size$-molecule associated with the ball $B$.
\end{lemma}

\begin{proof}
Let $A$ be a $(p(\cdot),\,\fz)$-atom associated with ball
$B:= B(x_B,\,r_B)\st\rn$ for some $x_B\in\rn$ and $r_B\in(0,\,\fz)$.
Then we know that
\begin{align}\label{eq x3.1}
\|A\|_{T^2(\rnn)}\le |B|^{\frac12}\|\chi_B\|_{\vp}^{-1}.
\end{align}
Let
\begin{equation}\label{eq xx}
m:=\pi_{M,\,L}(A)\ \ \text{and}\ \ b:=L^{-M}(m).
\end{equation}

Next, we show that $m$ is a $\size$-molecule associated with $B$, up to a harmless constant multiple.
Indeed, when $k\in\{0,\,\ldots,\,M\}$, by \eqref{eq xx} and \eqref{eq pi}, we find that,
for any $g\in L^2(\rn)$,
\begin{align}\label{eq x3.0}
&\int_\rn(r_B^2L)^k(b)(x)\ov{g(x)}\,dx\noz\\
&\hs=\int_\rn\int_0^\fz r_B^{2k}t^{2(M+1)}L^{k+1}e^{-t^2L}(A(\cdot,\,t))(x)\ov{g(x)}\,\frac{dt}{t}\,dx\noz\\
&\hs=\int_0^\fz\int_\rn r_B^{2k}t^{2(M+1)}A(x,\,t)
\ov{{(L^\ast)}^{k+1}e^{-t^2L^\ast}(g)(x)}\,dx\,\frac{dt}{t}.
\end{align}
From this, the fact that $\supp A\st\wh{B}$, Lemma \ref{lem 0},
Remark \ref{rem Assu}(ii), \eqref{eq-1} and \eqref{eq x3.1},
we deduce that,
for any $k\in\{0,\,\ldots,\,M\}$ and $g\in L^2(\rn)$,
\begin{align}\label{eq x3.2}
\lf|\int_\rn(r_B^2L)^k(b)(x)\ov{g(x)}\,dx\r|
&\le r_B^{2M}\iint_{\wh{B}}|A(x,\,t)|\lf|(t^2L^\ast)^{k+1}e^{-t^2L^\ast}(g)(x)\r|\,dx\,\frac{dt}{t}\noz\\
&\le r_B^{2M}\|A\|_{T^2(\rnn)}\lf\|(t^2L^\ast)^{k+1}e^{-t^2L^\ast}(g)\r\|_{T^{2}(\rnn)}\noz\\
&= r_B^{2M}\|A\|_{T^2(\rnn)}\|S_{L^\ast,\,k+1}(g)\|_{L^{2}(\rn)}\noz\\
&\ls r_B^{2M}|B|^{\frac12}\|\chi_B\|_{\vp}^{-1}\|g\|_{L^{2}(\rn)},
\end{align}
which implies that
\begin{align*}
\lf\|(r_B^2L)^k(b)\r\|_{L^2(\rn)}\ls r_B^{2M}|B|^{\frac12}\|\chi_B\|_{\vp}^{-1}.
\end{align*}
By this and \eqref{eq xx}, we conclude that, for any $j\in\{0,\,1,\,2\}$,
\begin{align}\label{eq x3.4}
\lf\|(r_B^{-2}L^{-1})^k(m)\r\|_{L^2(U_j(B))}\ls|B|^{\frac12}\|\chi_B\|_{\vp}^{-1}.
\end{align}

When $k\in\{0,\,\ldots,\,M\}$, from \eqref{eq x3.0}, Remark \ref{rem Assu}(ii),
\eqref{eq-1} and \eqref{eq x3.1},
we deduce that, for any $j\in\zz_+\cap[3,\,\fz)$ and
$g\in L^2(\rn)$ with $\supp g\st\ujb$,
\begin{align}\label{eq x3.3}
&\lf|\int_{\ujb}(r_B^2L)^k(b)(x)\ov{g(x)}\,dx\r|\noz\\
&\hs\le r_B^{2M}\iint_{\wh{B}}|A(x,\,t)|\lf|(t^2L^\ast)^{k+1}e^{-t^2L^\ast}(g)(x)\r|\,dx\,\frac{dt}{t}\noz\\
&\hs\le
r_B^{2M}\|A\|_{T^2(\rnn)}\lf\|(t^2L^\ast)^{k+1}e^{-t^2L^\ast}(g)\chi_{\wh{B}}\r\|_{T^{2}(\rnn)}\noz\\
&\hs\ls r_B^{2M}|B|^{\frac12}\|\chi_B\|_{\vp}^{-1}
\lf\|(t^2L^\ast)^{k+1}e^{-t^2L^\ast}(g)\chi_{\wh{B}}\r\|_{T^{2}(\rnn)}.
\end{align}
By the H\"{o}lder inequality and Remark \ref{rem Assu}(ii),
we find that
\begin{align*}
&\lf\|(t^2L^\ast)^{k+1}e^{-t^2L^\ast}(g)\chi_{\wh{B}}\r\|_{T^{2}(\rnn)}\\
&\hs=\lf[\int_{\rn}\iint_{\bgz(x)}\lf|(t^2L^\ast)^{k+1}e^{-t^2L^\ast}(g)(y)
\chi_{\wh{B}}(y,\,t)\r|^2\dydt\,dx\r]^{\frac{1}{2}}\\
&\hs\le\lf[\int_{B}\int_0^{r_B}\int_{B(x,\,t)\cap B}
\lf|(t^2L^\ast)^{k+1}e^{-t^2L^\ast}(g)(y)\r|^2\dydt\,dx\r]^{\frac{1}{2}}\\
&\hs\le\lf[\int_0^{r_B}\int_B\lf|(t^2L^\ast)^{k+1}e^{-t^2L^\ast}(g)(y)\r|^2
\,dy\,\frac{dt}{t}\r]^{\frac12}\\
&\hs\ls\lf[\int_0^{r_B}e^{-2c\frac{(2^jr_B)^2}{t^2}}
\|g\|^2_{L^2(\ujb)}\,\frac{dt}{t}\r]^{\frac12}\\
&\hs\ls\lf[\int_0^{r_B}\lf(\frac{t}{2^jr_B}\r)^{2N}
\,\frac{dt}{t}\r]^{\frac12}\|g\|_{L^{2}(\ujb)}\ls 2^{-jN}
\|g\|_{L^2(\ujb)},
\end{align*}
where the positive constant $c$ is as in \eqref{eq-dg} and $N\in\nn$ is determined below.
From this and \eqref{eq x3.3}, it follows that, for any $j\in\zz_+\cap[3,\,\fz)$ and
$g\in L^2(\rn)$ with $\supp g\st \ujb$,
\begin{align*}
\lf|\int_{\ujb}(r_B^2L)^k(b)(x)\ov{g(x)}\,dx\r|\ls 2^{-jN}r_B^{2M}|B|^{\frac12}\|\chi_B\|_{\vp}^{-1}\|g\|_{L^{2}(\ujb)}.
\end{align*}
This further implies that, for any $j\in\zz_+\cap[3,\fz)$
\begin{align*}
\lf\|(r_B^2L)^k(b)\r\|_{L^2(\ujb)}\ls
2^{-j(N+\frac{n}2)}r_B^{2M}|2^jB|^{\frac12}\|\chi_B\|_{\vp}^{-1}.
\end{align*}
By choosing some $N\in\nn$ such that $N+n/2>\varepsilon$ and \eqref{eq xx}, we conclude that,
for any $j\in\zz_+\cap[3,\fz)$,
\begin{align}\label{3.31x}
\lf\|(r_B^{-2}L^{-1})^k(m)\r\|_{L^2(\ujb)}\ls 2^{-j\varepsilon}|2^jB|^{\frac12}\|\chi_B\|_{\vp}^{-1}.
\end{align}

Combining \eqref{3.31x} and \eqref{eq x3.4}, we know that $m=\pi_{M,\,L}(A)$
is a $(p(\cdot),\,M,\,\varepsilon)_L$-molecule associated with $B$, up to a harmless constant
multiple. This finishes the proof of Lemma \ref{lem 5}.
\end{proof}

\begin{proposition}\label{pro 2}
Let $L$ satisfy Assumptions \ref{as-a} and \ref{as-b}
and $p(\cdot)\in C^{\log}(\rn)$ with $p_+\in(0,\,1]$.
Assume $M\in\nn$ and $\vez\in(0,\,\fz)$. Then, for any $f\in \vhp\cap L^2(\rn)$,
there exist $\{\lz_j\}_{j\in\nn}\st\cc$ and a family $\{m_j\}_{j\in\nn}$ of
$\size$-molecules, associated with balls $\{B_j\}_{j\in\nn}$ of $\rn$, such that
$f=\sum_{j=1}^\fz \lz_jm_j$ in $L^2(\rn)$ and
\begin{align*}
\ca(\{\lz_j\}_{j\in\nn},\,\{B_j\}_{j\in\nn})\le C\|f\|_{\vhp},
\end{align*}
where the positive constant $C$ is independent of $f$.
\end{proposition}
\begin{proof}
For any $f\in\vhp\cap L^2(\rn)$ and $(x,\,t)\in\rnn$,
let $F(x,\,t):=t^2L e^{-t^2L}(f)(x)$.
By \cite[Theorem F]{adm96}, we know that $t^2Le^{-t^2L}$ is bounded from $L^2(\rn)$
to $T^2(\rnn)$.
This, together with $f\in\vhp$, implies that $F\in T^2(\rnn)\cap T^{p(\cdot)}(\rnn)$.
Then, by Lemma \ref{lem 2} and Remark \ref{rem 2}(i), we conclude that
there exist $\{\lz_j\}_{j\in\nn}\st\cc$ and a family $\{a_j\}_{j\in\nn}$ of
$(p(\cdot),\,\fz)$-atoms, associated with balls $\{B_j\}_{j\in\nn}$ of $\rn$, such that
\begin{equation}\label{eq x3.5}
F=\sum_{j=1}^\fz \lz_j a_j\ \ \ \text{in}\ \ \ T^2(\rnn)\cap T^{p(\cdot)}(\rnn)
\end{equation}
and
\begin{align}\label{eq x3.5x}
\ca(\{\lz_j\}_{j\in\nn},\,\{B_j\}_{j\in\nn})\ls\|F\|_{T^{p(\cdot)}(\rnn)}\sim\|f\|_{\vhp}.
\end{align}
By the bounded holomorphic functional calculi for $L$, we know that
\begin{equation}\label{eq x3.6}
f=C_{(M)}\int_0^\fz(t^2L)^{M+1}e^{-t^2L}\lf(t^2Le^{-t^2L}(f)\r)\,\frac{dt}{t}
=\pi_{M,\,L}(F)\ \ \ \text{in}\ \ \ L^2(\rn),
\end{equation}
where $C_{(M)}$ is a positive constant such that
$C_{(M)}\int_0^\fz t^{2(M+2)}e^{-2t^2}\,\frac{dt}{t}=1.$
Via some arguments similar to those used in the proofs of \eqref{eq x3.0} and \eqref{eq x3.2},
we conclude that $\pi_{M,\,L}$ is bounded from $T^2(\rnn)$ to $L^2(\rn)$
(see also \cite[Proposition 4.5(i)]{bckyy13a}).
From this, \eqref{eq x3.6} and \eqref{eq x3.5}, it follows that
\begin{align}\label{eq x3.7}
f=C_{(M)}\pi_{M,\,L}\lf(\sum_{j=1}^\fz\lz_ja_j\r)
=C_{(M)}\sum_{j=1}^\fz\lz_j\pi_{M,\,L}(a_j)\ \ \ \text{in}\ \ \ L^2(\rn).
\end{align}
Noticing that, for any $M\in\nn$, $\vez\in(0,\,\fz)$ and $j\in\nn$,
$\pi_{M,\,L}(a_j)$ is a $\size$-molecule, up to a harmless constant multiple
(see Lemma \ref{lem 5}), by Definition \ref{def m-Hardy}, we know that
\eqref{eq x3.7} is a $(p(\cdot),\,M,\,\vez)$-molecular representation of $f$.
This, together with \eqref{eq x3.5x}, finishes the proof of Proposition \ref{pro 2}.
\end{proof}

Let $L$ satisfy Assumptions \ref{as-a} and \ref{as-b}
and $p(\cdot)\in \cp(\rn)$ with $p_+\in(0,\,1]$.
For any $M\in\nn$ and $\vez\in(0,\,\fz)$,
define $H_{L,\,{\rm fin},\,M}^{p(\cdot),\,\vez}(\rn)$ as the \emph{set of all
finite linear combinations of $\size$-molecules}.

We have the following proposition which plays a key role
in the proof of Theorem \ref{thm 2} below.
\begin{proposition}\label{pro 3}
Let $L$ satisfy Assumptions \ref{as-a} and \ref{as-b}
and $p(\cdot)\in \cp(\rn)$ with $p_+\in(0,\,1]$.
Assume $M\in\nn$ and $\vez\in(0,\,\fz)$.
Then $H_{L,\,{\rm fin},\,M}^{p(\cdot),\,\vez}(\rn)$ is dense in $\mhp$ with respect to the quasi-norm
$\|\cdot\|_{\mhp}$.
\end{proposition}

\begin{proof}
Let $g\in\mhp$. Then, by Definition \ref{def m-Hardy},
we know that, for any $\delta\in(0,\,\fz)$, there exists
a function $f\in\mol$ such that
\begin{align}\label{3.35x}
\|g-f\|_{\mhp}\le \delta/2.
\end{align}
By the definition of $\mol$, we find that there exist $\{\lz_j\}_{j\in\nn}\st\cc$
and a family $\{m_j\}_{j\in\nn}$ of $\size$-molecules, associated with balls
$\{B_j\}_{j\in\nn}$ of $\rn$, such that $f=\sum_{j=1}^\fz \lz_jm_j$ in $L^2(\rn)$
and $\ca(\{\lz_j\}_{j\in\nn},\,\{B_j\}_{j\in\nn})<\fz$.
Now, for any $N\in\nn$, let $f_N:=\sum_{j=1}^N \lz_jm_j$.
Then we have
\begin{align}\label{eq x3.8}
\|f-f_N\|_{\mhp}
&\le\ca\lf(\{\lz\}_{j=N+1}^\fz,\,\{B_j\}_{j=N+1}^\fz\r)\noz\\
&=\lf\|\sum_{j=N+1}^\fz\lf[\frac{|\lz_j|\chi_{B_j}}{\|\chi_{B_j}\|_{\vp}}\r]^{p_-}
\r\|_{L^{\frac{p(\cdot)}{p_-}}(\rn)}^{\frac{1}{p_-}}.
\end{align}
Since
\begin{align*}
\ca(\{\lz_j\}_{j\in\nn},\,\{B_j\}_{j\in\nn})
=\lf\|\sum_{j=1}^\fz\lf[\frac{|\lz_j|\chi_{B_j}}{\|\chi_{B_j}\|_{\vp}}\r]^{p_-}
\r\|_{L^{\frac{p(\cdot)}{p_-}}(\rn)}^{\frac{1}{p_-}}<\fz,
\end{align*}
it follows that, for almost every $x\in\rn$,
\begin{equation*}
\lim_{N\to\fz}\sum_{j=N+1}^\fz\lf[\frac{|\lz_j|\chi_{B_j}(x)}{\|\chi_{B_j}\|_{\vp}}\r]^{p_-}=0.
\end{equation*}
Combining this and the dominated convergence theorem
(see, for example, \cite[Lemma 3.2.8]{dhhr11}), we have
\begin{align*}
\lim_{N\to\fz}\lf\|\sum_{j=N+1}^\fz\lf[\frac{|\lz_j|\chi_{B_j}}{\|\chi_{B_j}\|_{\vp}}\r]^{p_-}
\r\|_{L^{\frac{p(\cdot)}{p_-}}(\rn)}^{\frac{1}{p_-}}=0.
\end{align*}
By this and \eqref{eq x3.8}, we conclude that $\|f-f_N\|_{\mhp}\to 0$ as $N\to\fz$.
Hence, we find that, for any $\delta\in(0,\,\fz)$, there exists some $N_0\in\nn$
such that, for any $N>N_0$,
\begin{align}\label{3.36x}
\|f-f_N\|_{\mhp}<\delta/2.
\end{align}
Obviously, for any $N\in\nn$, $f_N\in H_{L,\,{\rm fin},\,M}^{p(\cdot),\,\vez}(\rn)$.
From \eqref{3.35x} and \eqref{3.36x}, we deduce that, for any $\delta\in(0,\,\fz)$, when $N>N_0$,
\begin{align*}
\|g-f_N\|_{\mhp}\le\|g-f\|_{\mhp}+\|f-f_N\|_{\mhp}<\delta.
\end{align*}
Thus, $H_{L,\,{\rm fin},\,M}^{p(\cdot),\,\vez}(\rn)$ is dense in $\mhp$
with respect to the quasi-norm $\|\cdot\|_{\mhp}$.
This finishes the proof of Proposition \ref{pro 3}.
\end{proof}

By Propositions \ref{pro 1} and \ref{pro 2},
we immediately conclude Theorem \ref{thm 1} below, which
establishes the molecular characterization of $\vhp$.
Since the proof is obvious, we omit the details.
\begin{theorem}\label{thm 1}
Let $L$ satisfy Assumptions \ref{as-a} and \ref{as-b}
and $p(\cdot)\in C^{\log}(\rn)$ with $p_+\in(0,\,1]$.
Assume $M\in\nn\cap(\frac n2[\frac{1}{p_-}-\frac12],\,\fz)$
and $\varepsilon\in(\frac{n}{p_-},\,\fz)$.
Then $\mhp$ and $\vhp$ coincide with equivalent quasi-norms.
\end{theorem}

\begin{remark}\label{rem-atom}
\begin{enumerate}
\item[(i)]
Notice that Hofmann et al. \cite[Theorem 4.1]{hlmmy11} established
the atomic characterization of the Hardy space $H^1_L(X)$ associated with
a non-negative self-adjoint operator $L$ (see also \cite[Theorem 5.1]{jy11}
for the atomic characterization of $H_L^p(X)$ with $p\in(0,\,1]$). In this article,
we can not obtain an atomic characterization of $\vhp$ similar to
\cite[Theorem 4.1]{hlmmy11} (or \cite[Theorem 5.1]{jy11}),
though we can establish the molecular characterization of $\vhp$ (see Proposition \ref{pro 2})
by using the atomic decomposition of tent spaces.
The intrinsic reason for this is that
the operator $L$ of this article may not be self-adjoint
which has been pointed out in the introduction of \cite{hlmmy11}.
More precisely, by Lemma \ref{lem 5}, we know that the operator $\pi_{M,\,L}$
only maps any $(p(\cdot),\,\fz)$-atom $A$ of $T^{p(\cdot)}(\rnn)$ into
a $\size$-molecule of $\vhp$, which has no compact support.
However, if the operator $L$ is non-negative self-adjoint,
by the finite speed propagation for the wave equation
(see \cite[Definition 3.3 and Lemma 3.5]{hlmmy11}),
we can further show that $\pi_{M,\,L}(A)$ has compact support and hence is
an atom of $\vhp$, the details being omitted.

\item[(ii)]
In particular, when $p(\cdot)\equiv p\in(0,\,1]$ is a constant and
$L$ satisfies Assumptions \ref{as-a} and \ref{as-b},
Theorem \ref{thm 1} coincides with \cite[Theorem 3.15]{dl13}
in the case when the underlying space $X:=\rn$.

\item[(iii)]
When $p(\cdot)\equiv 1$ and $L$ is a one-to-one non-negative self-adjoint
operator, from Theorem \ref{thm 1}, we deduce that, for any given
$M\in\nn\cap (\frac{n}4,\,\fz)$ and $\varepsilon\in(n,\,\fz)$,
$H_{L,\,M}^{1,\,\varepsilon}(\rn)$ and $H_L^1(\rn)$ coincide with equivalent quasi-norms,
which was already obtained in \cite[Corollary 5.3]{hlmmy11}
and the ranges of $M$ and $\varepsilon$ coincide with
those of \cite[Corollary 5.3]{hlmmy11}.
Moreover, when $p(\cdot)\equiv p\in(0,\,1]$, Theorem \ref{thm 1} was already obtained in
\cite[Theorem 5.1]{jy11}.

\item[(iv)]
If $p(\cdot)\equiv p\in(0,\,1]$ is a constant and $L$ is the second order divergence
form elliptic operator as in \eqref{eq op},
by Theorem \ref{thm 1}, we find that, for any given $M\in\nn\cap(\frac{n}2[\frac1p-\frac12],\,\fz)$
and $\varepsilon\in(\frac{n}p,\,\fz)$, $H_{L,\,M}^{p,\,\varepsilon}(\rn)$
and $H_L^p(\rn)$ coincide with equivalent quasi-norms.
This is just \cite[Theorem 3.5]{hmm11} and the ranges of $M$ and $\varepsilon$
coincide with those of \cite[Theorem 3.5]{hmm11}.
\end{enumerate}
\end{remark}

\begin{corollary}\label{cor 2}
Let $L$ satisfy Assumptions \ref{as-a} and \ref{as-b}
and $p(\cdot)\in C^{\log}(\rn)$ with $p_+\in(0,\,1]$.
Suppose $T$ is a linear operator, or a positive sublinear operator, which
is bounded on $L^2(\rn)$.
Let $M\in\nn\cap(\frac n2[\frac{1}{p_-}-\frac12],\,\fz)$
and $\varepsilon\in(\frac{n}{p_-},\,\fz)$.
Assume that there exist positive constants $C$ and $\tz\in(\frac{n}{p_-},\,\fz)$
such that, for any $(p(\cdot),\,M,\,\vez)_L$-molecule $m$, associated with ball $B$ of $\rn$,
and $j\in\zz_+$,
\begin{align*}
\|T(m)\|_{L^2(\ujb)}\le C2^{-j\tz}|2^jB|^{\frac12}\|\chi_B\|_{\vp}^{-1}.
\end{align*}
Then there exists a positive constant $C$ such that, for any $f\in\vhp$,
\begin{align}\label{eq cor-2}
\|T(f)\|_{\vp}\le C\|f\|_{\vhp}.
\end{align}
\end{corollary}

\begin{proof}
By Theorem \ref{thm 1}, we know that
$\mol$ is dense in $\vhp$ with respect to the quasi-norm $\|\cdot\|_{\vhp}$.
Hence, to complete the proof of Corollary \ref{cor 2}, we only need to show
that, for all $f\in\mol$, \eqref{eq cor-2} holds true.
The remainder of the proof of Corollary \ref{cor 2} is a complete analogue of the proof
of Proposition \ref{pro 1}, the details being omitted.
This finishes the proof of Corollary \ref{cor 2}.
\end{proof}

\section{The duality of $\vhp$\label{s4}}
\hskip\parindent
Let $L$ satisfy Assumptions \ref{as-a} and \ref{as-b}.
In this section, we mainly consider the duality of $\vhp$.
To this end, motivated by \cite{hm09,hmm11}, we introduce the
following ${\rm BMO}$-type space $\dbmo$.
Here and hereafter, we denote by $L^\ast$ the \emph{adjoint operator} of $L$.

Let $p(\cdot)\in \cp(\rn)$ with $p_+\in(0,\,1]$ and
$L$ satisfy Assumptions \ref{as-a} and \ref{as-b}.
In what follows, let $\vec{0}_n$ be the \emph{origin} of $\rn$.
For any $M\in\nn$ and $\vez\in(0,\,\fz)$, define
\begin{align*}
\test:=\lf\{\mu:=L^M(\nu):\ \nu\in D(L^M),\ \|\mu\|_{\test}<\fz\r\},
\end{align*}
where $D(L^M)$ denotes the domain of $L^M$ and
\begin{align}\label{eq test}
\|\mu\|_{\test}
&:=\sup_{j\in\zz_+}2^{j\vez}\lf|B(\vec{0}_n,\,2^j)\r|^{-\frac12}\|\chi_{B(\vec{0}_n,\,1)}\|_{\vp}\noz\\
&\hs\times\sum_{k=0}^M\lf\|L^{-k}(\mu)\r\|_{L^2(U_j(B(\vec{0}_n,\,1)))}.
\end{align}
Let
\begin{align*}
\tez:=\bigcap_{\vez\in (0,\fz)}\lf(\test\r)^{\ast}.
\end{align*}
Here and hereafter, $(\test)^\ast$ denotes the \emph{dual space} of $\test$,
namely, the set of all the bounded linear functionals on $\vhp$ and, for any $f\in(\test)^\ast$
and $g\in\test$,
$\langle f,\,g\rangle_{\cm}$ denotes the duality between $(\test)^\ast$ and $\test$.

\begin{definition}\label{def bmo}
Let $p(\cdot)\in \cp(\rn)$ with $p_+\in(0,\,1]$, $M\in\nn$
and $L$ satisfy Assumptions \ref{as-a} and \ref{as-b}.
An element $f\in\tez$ is said to belong to $\dbmo$ if
\begin{align}\label{eq bmo}
\|f\|_{\dbmo}:=\sup_{B\st\rn}\frac{|B|^{\frac12}}{\|\chi_B\|_{\vp}}
\lf[\int_B\lf|\lf(I-e^{-r^2_BL^\ast}\r)^M(f)(x)\r|^2\,dx\r]^{\frac12}<\fz,
\end{align}
where the supremum is taken over all balls of $\rn$.
\end{definition}

\begin{remark}\label{rem bmo}
\begin{enumerate}
\item[(i)] We point out that \eqref{eq bmo} is well defined.
Indeed, since $\{e^{-tL}\}_{t>0}$ satisfies Assumption \ref{as-b},
it is easy to see that, for any ball $B\st\rn$, $\phi\in L^2(B)$,
$\vez\in(0,\,\fz)$ and $M\in\nn$, $(I-e^{-t^2L})^M(\phi)\in\test$.
For any $f\in\tez$, define
\begin{align}\label{eq 4.x0}
\lf\langle\lf(I-e^{-t^2L^\ast}\r)^M(f),\,\phi\r\rangle
:=\lf\langle f,\, \lf(I-e^{-t^2L}\r)^M(\phi)\r\rangle_\cm.
\end{align}
Then we know that there exists a positive constant $C_{(t,\,B)}$,
depending on $t$, $r_B$ and $\dist(B,\,B(\vec{0}_n,\,1))$, such that
\begin{align*}
\lf|\lf\langle\lf(I-e^{-t^2L^\ast}\r)^M(f),\,\phi\r\rangle\r|
&\le \|f\|_{(\test)^\ast}\lf\|\lf(I-e^{-t^2L}\r)^M(\phi)\r\|_{\test}\\
&\le C_{(t,\,B)}\|f\|_{(\test)^\ast}\|\phi\|_{L^2(B)}.
\end{align*}
By the Riesz theorem, we further conclude that, for any ball $B\st\rn$ and
$t\in(0,\,\fz)$,
$$(I-e^{-t^2L^\ast})^M(f)\in L^2(B)$$
and
\begin{align*}
\lf\langle\lf(I-e^{-t^2L^\ast}\r)^M(f),\,\phi\r\rangle
=\int_B (I-e^{-t^2L^\ast})^M(f)(x)\phi(x)\,dx.
\end{align*}
Thus, \eqref{eq bmo} is well defined.

\item[(ii)] An element $f\in\mathcal{M}_{p(\cdot),\,L^\ast}^{M,\,\ast}(\rn)$ is said to
belong to $\bmo$ if it satisfies \eqref{eq bmo} with $L^\ast$ replaced by $L$.
\end{enumerate}
\end{remark}

The following proposition shows that elements of $\test$ are just
$(p(\cdot),\,M,\,\vez)_L$-molecules of $\vhp$ and vice versa.
\begin{proposition}\label{pro 4}
Let $p(\cdot)\in C^{\log}(\rn)$ with $p_+\in(0,\,1]$,
$\vez\in(0,\,\fz)$ and $M\in\nn$.
If $\mu\in\test$, then $\mu$ is a harmless positive constant multiple of a
$(p(\cdot),\,M,\,\vez)_L$-molecule
associated with the ball $B(\vec{0}_n,\,1)$.
Conversely, if $m$ is a $(p(\cdot),\,M,\,\vez)_L$-molecule, associated with
ball $B\st\rn$, then $m\in\test$.
\end{proposition}

\begin{proof}
If $\mu\in\test$, then, by \eqref{eq test}, we find that, for any $j\in\zz_+$ and
$k\in\{0,\,\ldots,\,M\}$,
\begin{align*}
\lf\|L^{-k}(\mu)\r\|_{L^2(U_j(B(\vec{0}_n,\,1)))}\ls 2^{-j\vez}\lf|2^jB(\vec{0}_n,\,1)\r|^{\frac12}
\|\chi_{B(\vec{0}_n,\,1)}\|_{\vp}^{-1},
\end{align*}
which implies that $\mu$ is a harmless positive constant multiple of
a $(p(\cdot),\,M,\,\vez)_L$-molecule associated with the
ball $B(\vec{0}_n,\,1)$.

Conversely, if $m$ is a $(p(\cdot),\,M,\,\vez)_L$-molecule associated with
ball $B:= B(x_B,\,r_B)\st\rn$ with $x_B\in\rn$ and $r_B\in(0,\,\fz)$,
then, by Definition \ref{def mol},
we know that, for any $j\in\zz_+$ and $k\in\{0,\,\ldots,\,M\}$,
\begin{equation}\label{eq 4.0}
\lf\|L^{-k}(m)\r\|_{L^2(\ujb)}\le 2^{-j\vez}r_B^{2k}\lf|2^jB\r|^{\frac12}\|\chi_B\|_{\vp}^{-1}.
\end{equation}

On the other hand, it is easy to see that there exist $l_1,\,l_2\in\nn$, depending on $B$,
such that
\begin{equation}\label{eq 4.1}
B(\vec{0}_n,\,1)\st B(x_B,\,2^{l_1}r_B)\ \ \ \text{and}\ \ \ B(x_B,\,r_B)\st B(\vec{0}_n,\,2^{l_2}).
\end{equation}
By this and Lemma \ref{lem 3}, we have
\begin{align*}
2^{-l_1\frac{n}{p_-}}\|\chi_{B(\vec{0}_n,\,1)}\|_{\vp}\ls\|\chi_{B(x_B,\,r_B)}\|_{\vp}
\ls 2^{l_2\frac{n}{p_-}}\|\chi_{B(\vec{0}_n,\,1)}\|_{\vp}.
\end{align*}
Combining this, \eqref{eq 4.0} and \eqref{eq 4.1},
we find that there exists a positive constant $C_{(l_1,\,l_2,\,B)}$,
depending on $l_1$, $l_2$ and $B$, such that,
for any $j\in\zz_+\cap[l_2+1,\,\fz)$ and $k\in\{0,\,\ldots,\,M\}$,
\begin{align*}
\lf\|L^{-k}(m)\r\|_{L^2(U_j(B(\vec{0}_n,\,1)))}
&\le\lf\|L^{-k}(m)\r\|_{L^2\lf(2^{j+l_1}B(x_B,\,r_B)\setminus 2^{j-1-l_2}B(x_B,\,r_B)\r)}\\
&\le\sum_{l=-l_2}^{l_1}2^{-(j+l)\vez}r_B^{2k}\lf|2^{j+l}B(x_B,\,r_B)\r|^{\frac12}
\|\chi_{B(x_B,\,r_B)}\|^{-1}_{\vp}\\
&\le C_{(l_1,\,l_2,\,B)}2^{-j\vez}\lf|2^j B(\vec{0}_n,\,1)\r|^{\frac12}
\|\chi_{B(\vec{0}_n,\,1)}\|^{-1}_{\vp}.
\end{align*}
Similarly, when $j\in\{0,\,\ldots,\,l_2\}$, it holds true that
\begin{align*}
\lf\|L^{-k}(m)\r\|_{L^2(U_j(B(\vec{0}_n,\,1)))}
\ls2^{-j\vez}\lf|2^j B(\vec{0}_n,\,1)\r|^{\frac12}\|\chi_{B(\vec{0}_n,\,1)}\|^{-1}_{\vp},
\end{align*}
where the implicit positive constant depends on $l_1$, $l_2$ and $B$.
Therefore, we obtain
$$\|m\|_{\test}<\fz.$$
This implies that $m\in\test$,
which completes the proof of Proposition \ref{pro 4}.
\end{proof}

To prove the main result of this section, we need following lemmas which are,
respectively, slight variants of \cite[Lemmas 8.1 and 8.4]{hm09}
(see also \cite[Lemmas 4.1 and 4.3]{jy10}), the details being omitted.
\begin{lemma}\label{lem 4.1}
Let $p(\cdot)\in C^{\log}(\rn)$ with $p_+\in(0,\,1]$ and $M\in\nn$.
Then $f\in\bmo$ is equivalent to that
\begin{align*}
\|f\|_{{\rm BMO}_{p(\cdot),\,L}^{M,\,{\rm res}}(\rn)}:=\sup_{B\st\rn}\frac{|B|^{\frac12}}{\|\chi_B\|_{\vp}}
\lf\{\int_B\lf|\lf[I-(I+r_B^2L)^{-1}\r]^M(f)(x)\r|^2\,dx\r\}^{\frac12}<\fz,
\end{align*}
where the supremum is taken over all balls of $\rn$.
Moreover, there exists a positive constant $C$ such that, for any $f\in\bmo$,
$$C^{-1}\|f\|_{\bmo}\le\|f\|_{{\rm BMO}_{p(\cdot),\,L}^{M,\,{\rm res}}(\rn)}
\le C\|f\|_{\bmo}.$$
\end{lemma}

\begin{lemma}\label{lem 4.2}
Let $p(\cdot)\in C^{\log}(\rn)$ with $p_+\in(0,\,1]$, $\wz{\vez},\,\vez\in(0,\,\fz)$,
$M\in\nn$ and $\wz{M}>M+\wz\vez+\frac{n}4$.
Suppose that $f\in\tez$ satisfies
\begin{equation}\label{eq 4.x}
\int_\rn\frac{|[I-(I+L^\ast)^{-1}]^M(f)(x)|^2}{1+|x|^{n+\wz\vez}}\,dx<\fz.
\end{equation}
Then, for any $(p(\cdot),\,\wz{M},\,\vez)_L$-molecule $m$, it holds
true that
\begin{equation*}
\langle f,\,m\rangle_\cm= C_{(M)}\iint_{\rnn}(t^2L^\ast)^M e^{-t^2L^\ast}(f)(x)
\ov{t^2Le^{-t^2L}(m)(x)}\,\frac{dx\,dt}{t},
\end{equation*}
where $C_{(M)}$ is a positive constant depending on $M$.
\end{lemma}

\begin{remark}\label{rem 4.1}
We point out that, for any $\wz\vez\in(n(1+\frac2{p_-}-\frac2{p_+}),\,\fz)$, $M\in\nn$
and $f\in\dbmo$, $f$ satisfies \eqref{eq 4.x}.
Indeed, by Lemma \ref{lem 4.1}, we obtain
\begin{equation}\label{eq 4.2}
\sup_{B\st\rn}\frac{|B|^{\frac12}}{\|\chi_B\|_{\vp}}
\lf\{\int_B\lf|\lf[I-(I+r_B^2L^\ast)^{-1}\r]^M(f)(x)\r|^2\,dx\r\}^{\frac12}<\fz.
\end{equation}
We write
\begin{align}\label{eq 4.x1}
&\int_\rn\frac{|[I-(I+L^\ast)^{-1}]^M(f)(x)|^2}{1+|x|^{n+\wz\vez}}\,dx\noz\\
&\hs=\sum_{j=0}^\fz\int_{U_j(B(\vec{0}_n,\,1))}
\frac{|[I-(I+L^\ast)^{-1}]^M(f)(x)|^2}{1+|x|^{n+\wz\vez}}\,dx\noz\\
&\hs\le\sum_{j=0}^\fz2^{-j(n+\wz\vez)}\int_{U_j(B(\vec{0}_n,\,1))}
\lf|\lf[I-(I+L^\ast)^{-1}\r]^M(f)(x)\r|^2\,dx.
\end{align}
For any $j\in\zz_+$, we choose a family $\{B_k\}_{k=1}^{c_n2^{jn}}$ of balls with radius
$r_{B_k}\equiv 1$, where the positive constant $c_n:=\lfloor n^{\frac{n}2}-2^{-n}\rfloor+1$,
such that, for any $k\in\{1,\,\ldots,\,c_n2^{jn}\}$,
\begin{align}\label{eq 4.3x}
B_k\st B(\vec{0}_n,\,\sqrt{n}2^j),\ \ \ U_j(B(\vec{0}_n,\,1))\st \bigcup_{k=1}^{c_n2^{jn}}B_k
\end{align}
and, for any $x\in\rn$,
$\sum_{k=1}^{c_n2^{jn}}\chi_{B_k}(x)\le 3.$
From this, \eqref{eq 4.2}, \eqref{eq 4.3x} and Lemma \ref{lem 3}, it follows that
\begin{align*}
&\lf\{\int_{U_j(B(\vec{0}_n,\,1))}\lf|\lf[I-(I+L^\ast)^{-1}\r]^M(f)(x)\r|^2\,dx\r\}^{\frac12}\\
&\hs\le\sum_{k=1}^{c_n2^{jn}}
\lf\{\int_{B_k}\lf|\lf[I-(I+L^\ast)^{-1}\r]^M(f)(x)\r|^2\,dx\r\}^{\frac12}\\
&\hs\ls\sum_{k=1}^{c_n2^{jn}}\|\chi_{B_k}\|_{\vp}\|f\|_{\dbmo}\\
&\hs\ls2^{jn(1+\frac1{p_-}-\frac1{p_+})}\|\chi_{B(\vec{0}_n,\,1)}\|_{\vp}\|f\|_{\dbmo}.
\end{align*}
Combining this, \eqref{eq 4.x1} and the fact that $\wz\vez\in(n(1+\frac2{p_-}-\frac2{p_+}),\,\fz)$, we have
\begin{align*}
\int_\rn\frac{|[I-(I+L^\ast)^{-1}]^M(f)(x)|^2}{1+|x|^{n+\wz\vez}}\,dx
&\ls\sum_{j=0}^\fz2^{-j(n+\wz\vez)}2^{2jn(1+\frac1{p_-}-\frac1{p_+})}\|f\|^2_{\dbmo}\\
&\ls\|f\|^2_{\dbmo}<\fz.
\end{align*}
Therefore, the above claim holds true.
\end{remark}

By Lemma \ref{lem 4.1}, we obtain the following technical lemma.
The proof of Lemma \ref{lem 4.3} is similar to that of \cite[Lemma 8.3]{hm09},
the details being omitted.
\begin{lemma}\label{lem 4.3}
Let $p(\cdot)\in C^{\log}(\rn)$ with $p_+\in(0,\,1]$ and $M\in\nn$.
Then there exists a positive constant $C$ such that, for any $f\in\bmo$,
\begin{align*}
\sup_{B\st\rn}\frac{|B|^{\frac12}}{\|\chi_B\|_{\vp}}
\lf[\iint_{\wh{B}}\lf|(t^2L)^M e^{-t^2L}(f)(x)\r|^2\,\frac{dx\,dt}{t}\r]^{\frac12}
\le C\|f\|_{\bmo},
\end{align*}
where the supremum is taken over all balls $B$ of $\rn$.
\end{lemma}

We are now ready to establish the duality between $\vhp$ and $\dbmo$.
In what follows, let $(\vhp)^\ast$ be the dual space of $\vhp$,
namely, the set of all bounded linear functionals on $\vhp$.
\begin{theorem}\label{thm 2}
Let $p(\cdot)\in C^{\log}(\rn)$ with $p_+\in(0,\,1]$,
$\vez\in(\frac{n}{p_-},\,\fz)$ and $M\in\nn\cap(\frac{n}{2}[\frac1{p_-}-\frac12],\,\fz)$.
Then $(\vhp)^\ast$ coincides with ${\rm BMO}_{p(\cdot),\,L^\ast}^M(\rn)$ in the following sense:

\begin{enumerate}
\item[{\rm (i)}] Let $g\in(\vhp)^\ast$.
Then $g\in {\rm BMO}_{p(\cdot),\,L^\ast}^M(\rn)$ and,
for any $f\in H_{L,\,{\rm fin},\,M}^{p(\cdot),\,2,\,\vez}(\rn)$,
it holds true that
$g(f)=\langle g,\,f\rangle_\cm.$
Moreover, there exists a positive constant $C$ such that, for any
$g\in(\vhp)^\ast$,
$$\|g\|_{{\rm BMO}_{p(\cdot),\,L^\ast}^M(\rn)}\le C\|g\|_{(\vhp)^\ast}.$$

\item[{\rm (ii)}] Conversely, let $g\in{\rm BMO}_{p(\cdot),\,L^\ast}^M(\rn)$.
Then, for any $f\in H_{L,\,{\rm fin},\,M}^{p(\cdot),\,2,\,\vez}(\rn)$,
the linear functional $l_g$, given by
$l_g(f):=\langle g,\,f\rangle_\cm,$
has a unique bounded extension to $\vhp$ and there exists a positive constant $C$
such that, for any $g\in{\rm BMO}_{p(\cdot),\,L^\ast}^M(\rn)$,
$$\|l_g\|_{(\vhp)^\ast}\le C\|g\|_{{\rm BMO}_{p(\cdot),\,L^\ast}^M(\rn)}.$$
\end{enumerate}
\end{theorem}

\begin{remark}
If $p(\cdot)\equiv p\in(0,\,1]$ is a constant and $L$ is a one-to-one non-negative
self-adjoint operator (respectively, a second order divergence form elliptic operator),
then Theorem \ref{thm 2} coincides with \cite[Theorem 4.1]{jy11} in the case
when the underlying space $\mathcal{X}:=\rn$ and the Orlicz function $\omega(t):=t^p$
for all $t\in[0,\,\fz)$ (respectively, with \cite[Theorem 4.1]{jy10}
in the case with the same aforementioned Orlicz function $\omega$).
\end{remark}

\begin{proof}[Proof of Theorem \ref{thm 2}]
We first prove (i). Let $g\in(\vhp)^\ast$.
Then, for any $f\in\vhp$, we have
\begin{align}\label{eq 4.3}
|g(f)|\le\|g\|_{(\vhp)^\ast}\|f\|_{\vhp}.
\end{align}
By Proposition \ref{pro 1}, we know that,
for any $\vez\in(\frac{n}{p_-},\,\fz)$ and $(p(\cdot),\,M,\,\vez)_L$-molecule $m$,
$$\|m\|_{\vhp}\ls 1.$$
From this and \eqref{eq 4.3}, it follows that, for any $(p(\cdot),\,M,\,\vez)_L$-molecule $m$,
\begin{align}\label{eq 4.4}
|g(m)|\ls\|g\|_{(\vhp)^\ast}.
\end{align}
On the other hand, by Proposition \ref{pro 4}, we find that,
for any $\mu\in\test$ with $\|\mu\|_{\test}=1$, $\mu$ is a
harmless positive constant multiple of a $(p(\cdot),\,M,\,\vez)_L$-molecule
associated with the ball $B(\vec{0}_n,\,1)$. Let
$\langle g,\,\mu\rangle :=g(\mu).$
This, together with \eqref{eq 4.4},
implies that $g\in(\test)^\ast$ for any $\vez\in(0,\,\fz)$.
Hence, $g\in\tez$ and
\begin{equation}\label{eq 4.4x}
\langle g,\,\mu\rangle_\cm=\langle g,\,\mu\rangle=g(\mu).
\end{equation}

Next, we show that
\begin{equation}\label{eq 2-7}
\|g\|_{\dbmo}\ls \|g\|_{(\vhp)^\ast}.
\end{equation}
We first claim that, for any $B\st\rn$, $\vz\in L^2(B)$ with
$\|\vz\|_{L^2(B)}=1$,
$$\frac{|B|^{\frac12}}{\|\chi_B\|_{\vp}}(I-e^{r_B^2L})^M(\vz)$$
is a harmless positive constant multiple of a $(p(\cdot),\,M,\,\vez)_L$-molecule.
If this claim holds true, then, by Proposition \ref{pro 4}, \eqref{eq 4.x0}, \eqref{eq 4.4x}
and \eqref{eq 4.4}, we conclude that, for any $\vz\in L^2(B)$ with
$\|\vz\|_{L^2(B)}=1$,
\begin{align*}
&\lf|\frac{|B|^{\frac12}}{\|\chi_B\|_{\vp}}\int_B
\lf(I-e^{-r_B^2L^\ast}\r)^M(g)(x)\vz(x)\,dx\r|\\
&\hs=\lf|\lf\langle g,\,\frac{|B|^{\frac12}}{\|\chi_B\|_{\vp}}
\lf(I-e^{-r_B^2L}\r)^M(\vz)\r\rangle_\cm\r|\ls\|g\|_{(\vhp)^\ast},
\end{align*}
which implies that, for any ball $B\st\rn$,
\begin{align*}
\frac{|B|^{\frac12}}{\|\chi_B\|_{\vp}}\lf[\int_B\lf|\lf(I-e^{-r_B^2L^\ast}\r)^M(g)(x)\r|^2\,dx\r]^{\frac12}
\ls\|g\|_{(\vhp)^\ast}.
\end{align*}
Thus, \eqref{eq 2-7} holds true.

Therefore, to prove \eqref{eq 2-7}, it remains to show the above claim. Indeed,
when $k\in\{0,\,\ldots,\,M\}$, by the Minkowski inequality
and Remark \ref{rem Assu}(iii), we find that,
for any $j\in\zz_+\cap[2,\,\fz)$,
\begin{align}\label{eq 4.5}
&\lf\|\frac{|B|^{\frac12}}{\|\chi_B\|_{\vp}}(r_B^{-2}L^{-1})^k
\lf(I-e^{-r_B^2L}\r)^M(\vz)\r\|_{L^2(\ujb)}\noz\\
&\hs=\frac{|B|^{\frac12}}{\|\chi_B\|_{\vp}}
\lf\|r_B^{-2k}\lf[
\int_0^{r_B}\cdots\int_0^{r_B}2^kt_1\cdots t_ke^{-(t_1^2+\cdots+t_k^2)L}\,
dt_1\cdots dt_k\r]\r.\noz\\
&\hs\hs\circ (I-e^{r_B^2L})^{M-k}(\vz)\bigg\|_{L^2(\ujb)}\noz\\
&\hs\le\frac{|B|^{\frac12}}{\|\chi_B\|_{\vp}}r_B^{-2k}
\lf[\int_0^{r_B}\cdots\int_0^{r_B}2^kt_1\cdots t_k\r.\noz\\
&\hs\hs\lf.\times\lf\|e^{-(t_1^2+\cdots+t_k^2)L}(I-e^{-r_B^2L})^{M-k}(\vz)\r\|_{L^2(\ujb)}
\,dt_1\cdots dt_k\r]\noz\\
&\hs\ls\frac{|B|^{\frac12}}{\|\chi_B\|_{\vp}}r_B^{-2k}
\int_0^{r_B}\cdots\int_0^{r_B}2^kt_1\cdots t_k e^{-c\frac{(2^jr_B)^2}{r_B^2}}\|\vz\|_{L^2(B)}\,
dt_1\cdots dt_k\noz\\
&\hs\ls\frac{|B|^{\frac12}}{\|\chi_B\|_{\vp}}e^{-c2^{2j}}
\ls\frac{|B|^{\frac12}}{\|\chi_B\|_{\vp}}2^{-j\vez}.
\end{align}
Similarly, when $k\in\{0,\,\ldots,\,M\}$,  we know that, for any $j\in\{0,\,1\}$,
\begin{align*}
\lf\|\frac{|B|^{\frac12}}{\|\chi_B\|_{\vp}}(r_B^{-2}L^{-1})^k
\lf(I-e^{-r_B^2L}\r)^M(\vz)\r\|_{L^2(\ujb)}
\ls|B|^{\frac12}{\|\chi_B\|_{\vp}^{-1}}.
\end{align*}
This, combined with \eqref{eq 4.5}, implies the above claim.

Next, we prove (ii).
To this end, we only need to show that, for any $g\in\dbmo$ and
$f\in H_{L,\,{\rm fin},\,M}^{p(\cdot),\,\vez}(\rn)$ with
$\vez\in(0,\,\fz)$ and $M\in\nn$,
\begin{align}\label{eq 4.6}
|\langle g,\,f\rangle_\cm|\ls\|g\|_{\dbmo}\|f\|_{\vhp}.
\end{align}
Indeed, since the space $H_{L,\,{\rm fin},\,M}^{p(\cdot),\,\vez}(\rn)$
is dense in $H_{L}^{p(\cdot)}(\rn)$ with respect to the quasi-norm $\|\cdot\|_{\vhp}$
(see Proposition \ref{pro 3} and Theorem \ref{thm 1}),
from \eqref{eq 4.6}, we deduce that the linear functional $l_g$ given by
$l_g(f):=\langle g,\,f\rangle_\cm,$
initially defined on $H_{L,\,{\rm fin},\,M}^{p(\cdot),\,\vez}(\rn)$,
has a unique bounded extension to $H_{L}^{p(\cdot)}(\rn)$ and
$$\|l_g\|_{(H_L^{p(\cdot)}(\rn))^\ast}\ls\|g\|_{\dbmo}.$$

To prove \eqref{eq 4.6}, let $f\in H_{L,\,{\rm fin},\,M}^{p(\cdot),\,\vez}(\rn)$.
Then it is easy to see that $f\in H_L^{p(\cdot)}(\rn)\cap L^2(\rn)$.
This, together with \eqref{eq-1}, implies that
$$t^2Le^{-t^2L}(f)\in T^{p(\cdot)}(\rnn)\cap T^2(\rnn).$$
By Lemma \ref{lem 2} and Remark \ref{rem 2}(i), we conclude that
there exist $\{\lz_j\}_{j=1}^\fz\st\cc$ and a family $\{a_j\}_{j=1}^\fz$ of
$(p(\cdot),\,\fz)$-atoms, associated with balls $\{B_j\}_{j\in\nn}$ of $\rn$, such that
\begin{align*}
t^2Le^{-t^2L}(f)=\sum_{j=1}^\fz \lz_j a_j\ \ \ \text{in}\ \ T^{p(\cdot)}(\rnn)\cap T^2(\rnn)
\end{align*}
and
\begin{align*}
\ca(\{\lz_j\}_{j\in\nn},\,\{B_j\}_{j\in\nn})\sim
\lf\|t^2Le^{-t^2L}(f)\r\|_{T^{p(\cdot)}(\rnn)}
\sim \|f\|_{\vhp}.
\end{align*}
From this, Lemma \ref{lem 4.2}, the H\"{o}lder inequality, the fact that $\{a_j\}_{j\in\nn}$
are $(p(\cdot),\,\fz)$-atoms, Lemma \ref{lem 4.3} and Remark \ref{rem 2}(ii),
we deduce that, for any $f\in H_{L,\,{\rm fin},\,M}^{p(\cdot),\,\vez}(\rn)$,
\begin{align*}
|\langle g,\,f\rangle_\cm|
&=\lf|C_M\iint_{\rnn}(t^2L^\ast)^Me^{-t^2L^\ast}(g)(x)\ov{t^2Le^{-t^2L}(f)(x)}\,\frac{dx\,dt}{t}\r|\\
&\ls \sum_{j=1}^\fz|\lz_j|\iint_{\rnn}\lf|(t^2L^\ast)^Me^{-t^2L^\ast}(g)(x)\r|
|a_j(x,\,t)|\,\frac{dx\,dt}{t}\\
&\ls \sum_{j=1}^\fz|\lz_j|\lf[\iint_{\wh{B_j}}\lf|(t^2L^\ast)^Me^{-t^2L^\ast}(g)(x)\r|^2\,
\frac{dx\,dt}{t}\r]^{\frac12}
\lf[\iint_{\wh{B_j}}|a_j(x,\,t)|^2\,\frac{dx\,dt}{t}\r]^{\frac12}\\
&\ls \sum_{j=1}^\fz|\lz_j|\|g\|_{\dbmo}\|\chi_{B_j}\|_{\vp}|B_j|^{-\frac12}\|a_j\|_{T^2(\rnn)}\\
&\ls \sum_{j=1}^\fz|\lz_j|\|g\|_{\dbmo}\ls \ca(\{\lz_j\}_{j\in\nn},\,\{B_j\}_{j\in\nn})\|g\|_{\dbmo}\\
&\sim \|f\|_{\vhp}\|g\|_{\dbmo},
\end{align*}
namely, \eqref{eq 4.6} holds true.
This finishes the proof of Theorem \ref{thm 2}.
\end{proof}

\section{Variable Hardy spaces associated with second order divergence form
elliptic operators\label{s5}}
\hskip\parindent
In this section, we study the variable Hardy spaces $\vhp$
associated with second order divergence form elliptic operators $L$ as in \eqref{eq op}.
By making good use of the special structure of the divergence form elliptic operator,
we establish the non-tangential maximal function characterizations of $\vhp$.
Moreover, we establish the boundedness of the associated fractional integrals and Riesz transforms on $\vhp$.

Since $L$ in \eqref{eq op} satisfies Assumptions \ref{as-a} and \ref{as-b} (see Remark \ref{rem Assu*}(i)),
a corresponding theory of the variable Hardy space $\vhp$ with $L$ as in \eqref{eq op},
including its molecular characterization, can be obtained as a special case of
all results presented in previous sections.
Moreover, by \eqref{eq-od}, we have the following observation.
\begin{remark}\label{rem 8}
Let $L$ be as in \eqref{eq op}.
By \cite[Lemma 2.6]{hm09}, we know that,
for any $p\in(p_-(L),\,p_+(L))$, the square function $S_{L,\,k}$, with $k\in\nn$,
in \eqref{eq-1*} is bounded on $L^p(\rn)$,
where the positive constants $p_-(L)$ and $p_+(L)$ are, respectively, as in \eqref{eq bound1}
and \eqref{eq bound2}.
\end{remark}

\subsection{Non-tangential maximal function characterization of $\vhp$\label{s5.1}}
\hskip\parindent
In this subsection, we establish the non-tangential maximal function characterization of $\vhp$
with $L$ as in \eqref{eq op}.
We begin with recalling some notions from \cite{hm09}.

For any $\az\in(0,\,\fz)$, the \emph{non-tangential maximal function $\cn_h^{(\az)}$},
associated with the heat semigroup generated by $L$, is defined by setting,
for any $f\in L^2(\rn)$ and $x\in\rn$,
\begin{align*}
\cn_h^{(\az)}(f)(x):=\sup_{(y,t)\in\bgz_{\az}(x)}\lf[\frac1{(\az t)^n}\int_{B(y,\,\az t)}
\lf|e^{-t^2L}(f)(z)\r|^2\,dz\r]^{\frac12},
\end{align*}
where $\Gamma_\az(x)$ is as in \eqref{eq bgz}.
In particular, when $\az=1$, we simply write $\cn_h$ instead of $\cn_h^{(\az)}$.

Similar to Definition \ref{def vhp}, we introduce the Hardy space $H_{\cn_h}^{p(\cdot)}(\rn)$
as follows.
\begin{definition}\label{def nhp}
Let $p(\cdot)\in \cp(\rn)$ satisfy $p_+\in(0,\,1]$ and $L$ be
the second order divergence form elliptic operator as in \eqref{eq op}.
The \emph{Hardy space} $\nhp$ is defined as the completion of the set
\begin{align*}
\lf\{f\in L^2(\rn):\ \|f\|_{\nhp}:=\|\cn_h(f)\|_{\vp}<\fz\r\}
\end{align*}
with respect to the quasi-norm $\|\cdot\|_{\nhp}$.
\end{definition}

The following theorem establishes the non-tangential maximal function characterization of $\vhp$.
\begin{theorem}\label{thm 3}
Let $p(\cdot)\in C^{\log}(\rn)$ satisfy $p_+\in(0,\,1]$ and $L$ be
the second order divergence form elliptic operator as in \eqref{eq op}.
Then $\vhp$ and $\nhp$ coincide with equivalent quasi-norms.
\end{theorem}

\begin{remark}\label{rem-maximal}
\begin{enumerate}
\item[(i)]
The proof of Theorem \ref{thm 3} divides into two steps.
Step 1 is to show $\nhp\st\vhp$ and step 2 is the proof of the inverse
inclusion. The proof of step 1
relies on some known results, from \cite{bckyy13a,yy13}, which are essentially
deduced from a good-$\lambda$ inequality for $\cn_h$ whose proof is mainly
based on the special structure of the operator $L=-\div(A\nabla)$,
namely, the divergence form, and on some particular PDE techniques
(for example, the Caccioppoli inequalities for the solutions of parabolic
and elliptic systems; see also \cite{hm09} for some details).
If $L$ is merely an abstract operator satisfying Assumptions \ref{as-a} and \ref{as-b},
by an argument similar to that used in step 2
(see also \cite[Section 5.3]{jy10} and \cite[Theorem 7.5]{bckyy13a}),
we can establish the inclusion $\vhp\st\nhp$
(see \cite[Proposition 4.7]{hlmmy11} for a similar result).
However, we do not know how to prove the inverse inclusion
without invoking the special structure of $L$, which is still open.

\item[(ii)] Recently, Song and Yan \cite{sy16} established the non-tangential
maximal function characterization, via the atomic characterization,
of Hardy spaces associated with
non-negative self-adjoint operators $\wz L$ having Gaussian upper bounds
(see Remark \ref{rem Assu*}(ii)),
which was further generalized to the variable Hardy spaces $H_{\wz L}^{p(\cdot)}(\rn)$ in \cite{zy15}.
Their proof depends on a modification of a technique due to
Calder\'{o}n \cite{cal77}, which is different from the technique used in the
setting of second order divergence elliptic operators
(see, for example, \cite{hm09,jy10}).

\item[(iii)] Notice that, in \cite[Section 7]{hm09}, Hofmann and Mayboroda
established equivalent characterizations of the Hardy spaces $H_L^1(\rn)$
associated with the second order divergence form elliptic operators $L$
via both $\cn_h$ and the non-tangential maximal function $\cn_P$ associated to the Poisson
semigroup $\{e^{-t\sqrt{L}}\}_{t>0}$, which is defined by setting, for any $f\in L^2(\rn)$
and $x\in\rn$,
\begin{align*}
\cn_P(f)(x):=\sup_{(y,\,t)\in\bgz(x)}\lf[\frac{1}{t^n}\int_{B(y,\,t)}
\lf|e^{-t\sqrt{L}}(f)(x)\r|^2\,dx\r]^{\frac12},
\end{align*}
where $\bgz(x)$ is as in \eqref{eq bgz} with $\az=1$ (see also \cite[Section 5]{jy10}
and \cite[Theorem 7.5]{bckyy13a}). Motivated by this, we can
define the Hardy spaces $H^{p(\cdot)}_{\cn_P}(\rn)$ in a way similar to
that used in Definition \ref{def nhp}. It is natural to ask whether or not
these spaces $\vhp$ and $H^{p(\cdot)}_{\cn_P}(\rn)$ coincide with equivalent quasi-norms.
More generally, if $L$ is an abstract operator satisfying Assumptions \ref{as-a}
and \ref{as-b}, motivated by \cite{hlmmy11,hm09,jy10,bckyy13a}, it is natural
to ask whether or not one can characterize $\vhp$ via the square function
$S_{P,\,L}$ associated with the Poisson semigroup $\{e^{-t\sqrt{L}}\}_{t>0}$.
To restrict the length of this article, we address these problems in another
forthcoming article.

\item[(iv)]
Particularly, if $p(\cdot)\equiv p\in(0,\,1]$ is a constant, then Theorem \ref{thm 3}
was already obtained in \cite[Theorem 5.2]{jy10}.
\end{enumerate}
\end{remark}

To prove Theorem \ref{thm 3}, we first recall some auxiliary functions introduced in \cite{hm09}.
For any $f\in L^2(\rn)$ and $x\in\rn$, let
\begin{align*}
\mr_h(f)(x):=\sup_{t\in(0,\,\fz)}\lf[\frac1{t^n}\int_{B(x,\,t)}
\lf|e^{-t^2L}(f)(y)\r|^2\,dy\r]^{\frac12}
\end{align*}
and
\begin{align*}
\wz{\cs}_h(f)(x):=\lf[\iint_{\bgz(x)}\lf|t\nabla e^{-t^2L}(f)(y)\r|^2\dydt\r]^{\frac12},
\end{align*}
where $\bgz_\az(x)$ is as in \eqref{eq bgz}.

Let $q\in[1,\,\fz)$. Recall that a non-negative and locally integrable function $w$ on $\rn$
is said to belong to the \emph{class $A_q(\rn)$ of Muckenhoupt weights}, denoted by $w\in A_q(\rn)$,
if, when $q\in(1,\fz)$,
\begin{align*}
A_q(w):=\sup_{B\st\rn}
\frac{1}{|B|}\int_B w(x)\,dx\lf\{\frac{1}{|B|}\int_B [w(x)]^{-\frac{1}{q-1}}\,dx\r\}^{q-1}<\fz
\end{align*}
or
\begin{align*}
A_1(w):=\sup_{B\st\rn}\frac{1}{|B|}\int_B w(x)\,dx\lf\{\essinf_{x\in B}w(x)\r\}^{-1}<\fz,
\end{align*}
where the suprema are taken over all balls $B$ of $\rn$.

We also need the following lemma which is
called the extrapolation theorem for $\vp$ (see, for example,
\cite[Theorem 1.3]{cfmp06} and \cite[Theorem 7.2.1]{dhhr11})
and plays a key role in the proof of Theorem \ref{thm 3}.
\begin{lemma}[\cite{dhhr11}]\label{lem extro}
Let $\cf$ be a family of pairs of measurable functions on $\rn$ and $\Omega\st\rn$ an open set.
Assume that, for some $p_0\in(0,\,\fz)$ and any $w\in A_1(\rn)$,
\begin{equation*}
\int_\Omega|f(x)|^{p_0}w(x)\,dx\le C_{(w)}\int_\Omega|g(x)|^{p_0}w(x)\,dx\ \ \
\text{for any}\ \ (f,\,g)\in\cf,
\end{equation*}
where the positive constant $C_{(w)}$ depends only on $A_1(w)$.
Let $p(\cdot)\in C^{\log}(\rn)$ such that $p_-\in(p_0,\,\fz)$.
Then there exists a positive constant $C$ such that, for any $(f,\,g)\in\cf$,
\begin{equation*}
\|f\|_{\vp}\le C\|g\|_{\vp}.
\end{equation*}
\end{lemma}

We are now in a position to prove Theorem \ref{thm 3}.
\begin{proof}[Proof of Theorem \ref{thm 3}]
We first prove that, for any $f\in\nhp\cap L^2(\rn)$,
\begin{align}\label{eq 5.0x}
\|f\|_{\vhp}\ls \|f\|_{\nhp}.
\end{align}

Indeed, by \cite[Lemma 5.2]{jy10} (see also \cite[Lemma 5.4]{hm09}), we
find that, for any $f\in L^2(\rn)$ and $x\in\rn$,
$S_L(f)(x)\ls\wz{\cs}_h(f)(x)$.
Hence, for any $f\in\nhp\cap L^2(\rn)$,
\begin{align}\label{eq 5.0}
\|S_L(f)\|_{\vp}\ls\lf\|\wz{\cs}_h(f)\r\|_{\vp}.
\end{align}
On the other hand, from \cite[p.\,116]{bckyy13a}, we deduce that,
for any $w\in A_1(\rn)$, there exists a positive constant $C_{(w)}$, depending on $A_1(w)$,
such that, for any $f\in \nhp\cap L^2(\rn)$,
\begin{align*}
\int_\rn \lf[\wz{\cs}_h(f)(x)\r]^{p_0}w(x)\,dx
\le C_{(w)}\int_\rn\lf[\cn_h(f)(x)\r]^{p_0}w(x)\,dx,
\end{align*}
where $p_0\in(0,\,p_-)$ and $p_-$ is as in \eqref{eq var}.
Combining this and Lemma \ref{lem extro}, we obtain
\begin{align*}
\lf\|\wz{S}_h(f)\r\|_{\vp}\ls \|\cn_h(f)\|_{\vp}.
\end{align*}
By this and \eqref{eq 5.0}, we find that,
for any $f\in \nhp\cap L^2(\rn)$,
\begin{align*}
\|S_L(f)\|_{\vp}\ls\|\cn_h(f)\|_{\vp}.
\end{align*}
This implies \eqref{eq 5.0x}. Therefore,
\begin{align*}
\lf[\nhp\cap L^2(\rn)\r]\st\lf[\vhp\cap L^2(\rn)\r].
\end{align*}

Next, we show the inverse inclusion.
To this end, it suffices to prove that, for any $f\in\vhp\cap L^2(\rn)$,
\begin{align}\label{eq 5.3z}
\|f\|_{\nhp}\ls \|f\|_{\vhp}.
\end{align}

Indeed, from \cite[p.\,117]{bckyy13a}, we deduce that,
for any $w\in A_1(\rn)$, there exists a positive constant $C_{(w)}$, depending on $A_1(w)$,
such that, for any $f\in \vhp\cap L^2(\rn)$,
\begin{align*}
\int_\rn \lf[\cn_h(f)(x)\r]^{p_0}w(x)\,dx
\le C_{(w)}\int_\rn\lf[\mr_h(f)(x)\r]^{p_0}w(x)\,dx,
\end{align*}
where $p_0\in(0,\,p_-)$.
From this and Lemma \ref{lem extro},
it follows that, for any $f\in \vhp\cap L^2(\rn)$,
\begin{align}\label{eq 5.3x}
\|\cn_h(f)\|_{\vp}\ls\|\mr_h(f)\|_{\vp}.
\end{align}

Now we prove that, for any $f\in \vhp\cap L^2(\rn)$,
\begin{align}\label{eq 5.3y}
\|\mr_h(f)\|_{\vp}\ls \|f\|_{\vhp}.
\end{align}
To this end, by the fact that $\mr_h$ is bounded on $L^2(\rn)$ (see \cite[p.\,82]{hm09})
and Corollary \ref{cor 2},
we know that it suffices to prove that,
for any given $M\in\nn\cap(\frac n2[\frac1{p_-}-\frac12],\fz)$
and $\varepsilon\in(\frac{n}{p_-},\,\fz)$,
there exists a positive constant
$\tz\in(\frac{n}{p_-},\,\fz)$ such that,
for any $(p(\cdot),\,M,\,\vez)_L$-molecule $m$,
associated with ball $B:= B(x_B,\,r_B)\st\rn$ with $x_B\in\rn$
and $r_B\in(0,\,\fz)$, and $j\in\zz_+$,
\begin{align}\label{eq 5.x}
\|\mr_h(m)\|_{L^2(\ujb)}\ls 2^{-j\tz}|2^j B|^{\frac12}\|\chi_B\|_{\vp}^{-1},
\end{align}
where, for each $j\in\zz_+$, $U_j(B)$ is as in \eqref{eq ujb}.

Indeed, when $j\in\{0,\,\ldots,\,10\}$, by the boundedness of $\mr_h$ on $L^2(\rn)$,
we know that, for any given $\tz\in(\frac{n}{p_-},\,\fz)$,
\begin{align*}
\|\mr_h(m)\|_{L^2(\ujb)}\le\|m\|_{L^2(\rn)}
\ls 2^{-j\tz}|2^j B|^{\frac12}\|\chi_B\|_{\vp}^{-1}.
\end{align*}

When $j\in\nn\cap[11,\,\fz)$, for any $x\in\ujb$, we write
\begin{align}\label{eq 5.3}
\mr_h(m)(x)
&\le\lf\{\sup_{t\in(0,\,2^{aj-2}r_B]}
+\sup_{t\in(2^{aj-2}r_B,\,\fz)}\r\}
\lf[\frac1{t^n}\int_{B(x,\,t)}\lf|
e^{-t^2L}(m)(y)\r|^2\,dy\r]^{\frac12}\\
&=:{\rm I}_j(x)+{\rm II}_j(x),\noz
\end{align}
where $a\in(0,\,1)$ is a positive constant to be fixed below.

To handle ${\rm I}_j$, let $S_j(B):=(2^{j+3}B)\setminus (2^{j-3}B)$,
\begin{align*}
R_j(B):=(2^{j+5}B)\setminus (2^{j-5}B)\ \ \ \text{and}\ \ \
E_j(B):=[R_j(B)]^\com.
\end{align*}
Then $m=m\chi_{R_j(B)}+m\chi_{E_j(B)}$.
When $t\in(0,\,2^{aj-2}r_B]$, it is easy to see that, for any $x\in\ujb$,
$B(x,\,t)\st S_j(B)$ and $\dist(S_j(B),\,E_j(B))\sim 2^{j}r_B$.
By this, Assumption \ref{as-b}
and the fact that $m$ is a $(p(\cdot),\,M,\,\vez)_L$-molecule, we conclude that
\begin{align}\label{eq 5.1}
&\lf\|\sup_{t\in(0,\,2^{aj-2}r_B]}\lf[\frac1{t^n}\int_{B(\cdot,\,t)}\lf|
e^{-t^2L}\lf(m\chi_{E_j(B)}\r)(y)\r|^2\,dy\r]^{\frac12}\r\|_{L^2(\ujb)}\noz\\
&\hs\le\lf\|\sup_{t\in(0,\,2^{aj-2}r_B]}\lf[\frac1{t^n}\int_{S_j(B)}\lf|
e^{-t^2L}\lf(m\chi_{E_j(B)}\r)(y)\r|^2\,dy\r]^{\frac12}\r\|_{L^2(\ujb)}\noz\\
&\hs\ls\lf\|\sup_{t\in(0,\,2^{aj-2}r_B]}t^{-\frac n2}e^{-c\frac{(2^jr_B)^2}{t^2}}
\|m\|_{L^2(E_j(B))}\r\|_{L^2(\ujb)}\noz\\
&\hs\ls\sup_{t\in(0,\,2^{aj-2}r_B]}t^{-\frac n2}\lf(\frac t{2^jr_B}\r)^N|2^jB|^{\frac12}
\|m\|_{L^2(E_j(B))}\noz\\
&\hs\ls 2^{-j[N(1-a)+\frac n2]}|2^jB|^{\frac12}\|\chi_B\|_{\vp}^{-1},
\end{align}
where $N\in\nn\cap(\frac n2,\,\fz)$ is fixed below.
By the fact that $\mr_h$ is bounded on $L^2(\rn)$, we obtain
\begin{align*}
&\lf\|\sup_{t\in(0,\,2^{aj-2}r_B]}\lf[\frac1{t^n}\int_{B(\cdot,\,t)}\lf|
e^{-t^2L}\lf(m\chi_{R_j(B)}\r)(y)\r|^2\,dy\r]^{\frac12}\r\|_{L^2(\ujb)}\\
&\hs\le\lf\|\mr_h\lf(m\chi_{R_j(B)}\r)\r\|_{L^2(\rn)}
\ls\|m\|_{L^2(R_j(B))}
\ls 2^{-j\vez}|2^j B|^{\frac12}\|\chi_B\|_{\vp}^{-1}.
\end{align*}
This, together with \eqref{eq 5.1}, implies that
\begin{align}\label{eq 5.4}
\|{\rm I}_j\|_{L^2(\ujb)}\ls
\lf\{2^{-j\vez}+2^{-j[N(1-a)+\frac n2]}\r\}|2^j B|^{\frac12}\|\chi_B\|_{\vp}^{-1}.
\end{align}

Now we consider the term ${\rm II}_j$. For any $j\in\nn\cap[11,\,\fz)$ and
$x\in\ujb$, we have
\begin{align}\label{eq 5.2}
{\rm II}_j(x)
&=\sup_{t\in(2^{aj-2}r_B,\,\fz)}\lf[\frac1{t^n}\int_{B(x,\,t)}\lf|
\lf(t^2L\r)^{M}e^{-t^2L}\lf(t^{-2M}L^{-M}(m)\r)(y)\r|^2\,dy\r]^{\frac12}\noz\\
&\ls 2^{-2aMj}\noz\\
&\hs\times\sup_{t\in(2^{aj-2}r_B,\,\fz)}\lf[\frac1{t^n}\int_{B(x,\,t)}\lf|
\lf(t^2L\r)^{M}e^{-t^2L}\lf(r_B^{-2M}L^{-M}(m)\r)(y)\r|^2\,dy\r]^{\frac12}\noz\\
&\ls 2^{-2aMj}\mr_h^{(M)}\lf(r_B^{-2M}L^{-M}(m)\r)(x),
\end{align}
where $\mr_h^{(M)}$ is defined by setting, for any $f\in L^2(\rn)$ and $x\in\rn$,
\begin{align*}
\mr_h^{(M)}(f)(x):=\sup_{t\in(0,\,\fz)}\lf[\frac1{t^n}\int_{B(x,\,t)}
\lf|\lf(t^2L\r)^{M}e^{-t^2L}(f)(y)\r|^2\,dy\r]^{\frac12}.
\end{align*}
From \eqref{eq 5.2}, the boundedness of $\mr_h^{(M)}$ on $L^2(\rn)$ (see \cite[p.\,82]{hm09})
and Remark \ref{r-0902}, we deduce that
\begin{align*}
\|{\rm II}_j\|_{L^2(\ujb)}
&\ls 2^{-2aMj}\lf\|\mr_h^{(M)}\lf(r_B^{-2M}L^{-M}(m)\r)\r\|_{L^2(\rn)}\\
&\ls 2^{-2aMj}\lf\|\lf(r_B^{-2}L^{-1}\r)^M(m)\r\|_{L^2(\rn)}
\ls 2^{-j(2aM+\frac n2)}|2^jB|^{\frac12}\|\chi_B\|_{\vp}^{-1}.
\end{align*}
Combining this, \eqref{eq 5.4} and \eqref{eq 5.3}, we find that,
for any $(p(\cdot),\,M,\,\vez)_L$-molecule $m$ and $j\in\nn\cap[11,\,\fz)$,
\begin{align*}
\|\mr_h(m)\|_{L^2(\ujb)}\ls \lf[2^{-j\vez}+2^{-j[N(1-a)+\frac n2]}+2^{-j(2aM+\frac n2)}\r]
|2^jB|^{\frac12}\|\chi_B\|_{\vp}^{-1}.
\end{align*}
Let
$$\tz:=\min\lf\{\vez,\,N(1-a)+\frac{n}2,\,2aM+\frac n2\r\}.$$
By fixing some $M\in\nn\cap(\frac n2[\frac1{p_-}-\frac12],\fz)$,
$a\in(0,\,1)$, $N\in\nn\cap(\frac n2,\,\fz)$
and $\vez\in(\frac{n}{p_-},\,\fz)$, we have
$\tz\in(\frac{n}{p_-},\,\fz)$.
Thus, we obtain \eqref{eq 5.x}, which further implies \eqref{eq 5.3y}.
By \eqref{eq 5.3y} and \eqref{eq 5.3x}, we conclude that \eqref{eq 5.3z} holds true.
This, together with \eqref{eq 5.0x} and
a density argument then finishes the proof of Theorem \ref{thm 3}.
\end{proof}

\subsection{Boundedness of fractional integrals $L^{-\az}$\label{s5.2}}
\hskip\parindent
In this subsection, we show that the fractional integrals $L^{-\az}$ is
bounded from $\vhp$ to $H_L^{q(\cdot)}(\rn)$.
We begin with recalling some notions and well known results.

Let $L$ be the second order divergence form elliptic operator as in \eqref{eq op}
and $\az\in(0,\,\frac{n}2)$.
Recall that the generalized fractional integral $L^{-\az}$ is defined
by setting, for any $f\in L^2(\rn)$ and $x\in\rn$,
\begin{equation*}
L^{-\az}(f)(x):=\frac{1}{\bgz(\az)}\int_0^\fz t^{\az-1}e^{-tL}(f)(x)\,dt.
\end{equation*}

\begin{remark}\label{rem 3}
Let $p_-(L)$ and $p_+(L)$ be, respectively, as in \eqref{eq bound1} and \eqref{eq bound2}.
Then, by \cite[Proposition 5.3]{au07}, we know that, for any
$p_-(L)<p<q<p_+(L)$ and $\az=\frac{n}{2}(\frac1p-\frac1q)$,
$L^{-\az}$ is bounded from $L^p(\rn)$ to $L^q(\rn)$.
\end{remark}

To establish the boundedness of $L^{-\az}$ on $\vhp$, we need the following
technical lemma, which is a slight modification of \cite[Lemma 5.2]{s13}
with cubes therein replaced by balls here.
The proof of Lemma \ref{lem 5.2} is direct, the details being omitted.

\begin{lemma}\label{lem 5.2}
Let $\eta\in(0,\,n)$ and $p(\cdot)\in C^{\log}(\rn)$ with $p_+\in(0,\,\frac n\eta)$.
Define $q(\cdot)\in C^{\log}(\rn)$ by setting, for all $x\in\rn$,
$\frac{1}{q(x)}:=\frac{1}{p(x)}-\frac{\eta}{n}.$
Then there exists a positive constant $C$ such that,
for any sequence $\{B_j\}_{j\in\nn}$ of balls in $\rn$ and $\{\lz_j\}_{j\in\nn}\st\cc$,
\begin{align*}
\lf\|\sum_{j\in\nn}|\lz_j||B_j|^{\frac{\eta}{n}}\chi_{B_j}\r\|_{L^{q(\cdot)}(\rn)}
\le C\lf\|\sum_{j\in\nn}|\lz_j|\chi_{B_j}\r\|_{\vp}.
\end{align*}
\end{lemma}

\begin{theorem}\label{thm 5}
Let $\az\in(0,\,\frac12]$ and
$p(\cdot),\,q(\cdot)\in C^{\log}(\rn)$ with $p_+,\,q_+\in(0,\,1]$.
Assume that, for any $x\in\rn$, it holds true that
$\frac{1}{q(x)}=\frac{1}{p(x)}-\frac{2\az}{n}.$
Then there exists a positive constant $C$ such that, for any $f\in \vhp$,
$\|L^{-\az}(f)\|_{H^{q(\cdot)}_L(\rn)}\le C\|f\|_{\vhp}.$
\end{theorem}

\begin{proof}
Since $\vhp\cap L^2(\rn)$ is dense in $\vhp$, to prove Theorem \ref{thm 5},
we only need to show that, for any $f\in\vhp\cap L^2(\rn)$,
\begin{align}\label{eq 5.xx}
\lf\|S_L(L^{-\az}(f))\r\|_{L^{q(\cdot)}(\rn)}\ls\|f\|_{\vhp}.
\end{align}

From Proposition \ref{pro 2}, we deduce that, for any $f\in\vhp\cap L^2(\rn)$,
$M\in\nn$ and $\vez\in(0,\,\fz)$,
there exist $\{\lz_j\}_{j\in\nn}\st\cc$ and a family $\{m_j\}_{j\in\nn}$ of
$(p(\cdot),\,M,\,\vez)_L$-molecules, associated with balls $\{B_j\}_{j\in\nn}$ of $\rn$,
such that
\begin{align}\label{eq 5.9}
f=\sum_{j=1}^\fz \lz_jm_j\ \ \ \text{in}\ \  L^2(\rn)
\end{align}
and
\begin{align}\label{eq 5.9x}
\ca(\{\lz_j\}_{j\in\nn},\,\{B_j\}_{j\in\nn})\ls\|f\|_{\vhp}.
\end{align}
Let
$\frac1s:=\frac12-\frac{2\az}n$. Then $s\in(2,\,\frac {2n}{n-2}]\st (p_-(L),\,p_+(L))$ due to the
fact that $\az\in(0,\frac12]$ and hence, by Remark \ref{rem 3},
we know that $L^{-\az}$ is bounded from $L^2(\rn)$
to $L^s(\rn)$.
This, together with \eqref{eq 5.9}, implies that
\begin{align}\label{eq 5.y0}
L^{-\az}(f)=\sum_{j=1}^\fz\lz_jL^{-\az}(m_j)\ \ \ \text{in}\ \ L^s(\rn).
\end{align}
By the fact $s\in(2,\,p_+(L))\st(p_-(L),\,p_+(L))$,
we find that $S_L$ is bounded on $L^s(\rn)$ (see Remark \ref{rem 8}).
Combining this, \eqref{eq 5.y0} and the Riesz theorem,
we know that there exists a subsequence of
$\{S_L(\sum_{j=1}^N \lz_jL^{-\az}(m_j))\}_{N\in\nn}$
(without loss of generality, we may use the same notation as the original sequence)
such that, for almost every $x\in\rn$,
$S_L(L^{-\az}(f))(x)=\lim_{N\to\fz}S_L(\sum_{j=1}^N \lz_j L^{-\az}(m_j))(x)$.
Thus, for almost every $x\in\rn$,
$S_L(L^{-\az}(f))(x)\le \sum_{j=1}^\fz|\lz_j|S_L(L^{-\az}(m_j))(x),$
which further implies that
\begin{align}\label{eq 5.y2}
\|S_L(L^{-\az}(f))\|_{L^{q(\cdot)}(\rn)}
&\le\lf\|\sum_{j=1}^\fz|\lz_j|S_L(L^{-\az}(m_j))\r\|_{L^{q(\cdot)}(\rn)}\noz\\
&=\lf\|\sum_{j=1}^\fz\sum_{k=0}^\fz|\lz_j|S_L(L^{-\az}(m_j))\chi_{U_k(B_j)}\r\|_{L^{q(\cdot)}(\rn)}.
\end{align}

To prove Theorem \ref{thm 5}, it suffices to show that
there exist some $M\in\nn$, $\vez\in(0,\,\fz)$
and a positive constant $\tz\in(\frac{n}{p_-},\,\fz)$ such that,
for any $(p(\cdot),\,M,\,\vez)_L$-molecule $m$, associated with a ball $B$ of $\rn$,
and $k\in\zz_+$,
\begin{align}\label{eq 5.y1}
\lf\|S_L(L^{-\az}(m))\r\|_{L^2(U_k(B))}\ls 2^{-k\tz}|2^kB|^{\frac1q}\|\chi_{B}\|_{\vp}^{-1},
\end{align}
where $\frac1q-\frac12=\frac{2\az}{n}$.
Indeed, if \eqref{eq 5.y1} holds true, then, by the H\"{o}lder inequality,
we find that, for any fixed $r\in(1,\,2)$, $j\in\nn$ and $k\in\zz_+$,
\begin{align*}
\lf\|S_L(L^{-\az}(m_j))\chi_{U_k(B_j)}\r\|_{L^r(\rn)}
&\ls|2^kB_j|^{\frac1r-\frac12}\lf\|S_L(L^{-\az}(m_j))\r\|_{L^2(U_k(B))}\\
&\ls 2^{-k\tz}|2^kB_j|^{\frac1r-\frac12}\frac{|2^kB_j|^{\frac1q}}{\|\chi_{B_j}\|_{\vp}}\\
&\sim 2^{-k\tz}|2^kB_j|^{\frac{2\az}{n}}
\frac{|2^kB_j|^{\frac1r}}{\|\chi_{B_j}\|_\vp},
\end{align*}
which implies that
\begin{align*}
\lf\|2^{k\tz}\frac{\|\chi_{B_j}\|_\vp}{|2^kB_j|^{\frac{2\az}{n}}}S_L(L^{-\az}(m_j))
\chi_{U_k(B_j)}\r\|_{L^r(\rn)}\ls |2^kB_j|^{\frac1r}.
\end{align*}
From this, \eqref{eq 5.y2}, Remark \ref{rem 1}(iii), the fact that $p_-<q_-\in(0,\,1)$,
Lemmas \ref{lem key} and \ref{lem 5.2}, it follows that
\begin{align}\label{eq 5.y3x}
&\|S_L(L^{-\az}(f))\|_{L^{q(\cdot)}(\rn)}\noz\\
&\hs\le\lf\|\lf\{\sum_{j=1}^\fz\sum_{k=0}^\fz\lf[|\lz_j|S_L(L^{-\az}(m_j))\chi_{U_k(B_j)}\r]^{q_-}
\r\}^{\frac1{q_-}}\r\|_{L^{q(\cdot)}(\rn)}\noz\\
&\hs\ls\lf\{\sum_{k=0}^\fz 2^{-k\tz q_-}\lf\|\sum_{j=1}^\fz\lf[\frac{|\lz_j|
|2^kB_j|^{\frac{2\az}{n}}}{\|\chi_{B_j}
\|_\vp}\chi_{2^kB_j}\r]^{q_-}\r\|_{L^{\frac{q(\cdot)}{q_-}}(\rn)}\r\}^{\frac1{q_-}}\noz\\
&\hs\ls\lf\{\sum_{k=0}^\fz 2^{-k\tz q_-}\lf\|\sum_{j=1}^\fz\lf[\frac{|\lz_j|}{\|\chi_{B_j}
\|_\vp}\chi_{2^kB_j}\r]^{q_-}\r\|_{L^{\frac{p(\cdot)}{q_-}}(\rn)}\r\}^{\frac1{q_-}}\noz\\
&\hs\ls\lf\{\sum_{k=0}^\fz 2^{-k\tz q_-}\lf\|\lf\{\sum_{j=1}^\fz\lf[\frac{|\lz_j|}
{\|\chi_{B_j}\|}\chi_{2^kB_j}\r]^{p_-}\r\}^{\frac{1}{p_-}}\r\|^{q_-}_{\vp}
\r\}^{\frac1{q_-}}.
\end{align}
Notice that, for any $x\in\rn$,
\begin{align}\label{eq 5.y5}
\chi_{2^kB_j}(x)\le 2^{kn}\cm(\chi_{B_j})(x).
\end{align}
By the fact that $\tz\in(\frac{n}{p_-},\,\fz)$, we can choose a positive constant
$r\in(0,\,p_-)$ such that $\tz\in(\frac{n}{r},\fz)$.
From this, \eqref{eq 5.y5}, \eqref{eq 5.y3x}, Remark \ref{rem 1}(iii)
and Lemma \ref{lem fs}, we deduce that
\begin{align}\label{eq 5.y5x}
&\|S_L(L^{-\az}(f))\|_{L^{q(\cdot)}(\rn)}\noz\\
&\hs\ls \lf\{\sum_{k=0}^\fz 2^{-k\tz q_-}\lf\|\lf\{\sum_{j=1}^\fz2^{\frac{knp_-}{r}}
\lf[\cm\lf(\frac{|\lz_j|^r}{\|\chi_{B_j}\|^r_\vp}\chi_{B_j}\r)\r]^{\frac{p_-}{r}}\r\}^{\frac{r}{p_-}}
\r\|_{L^{\frac{p(\cdot)}{r}}}^{\frac{q_-}{r}}\r\}^{\frac1{q_-}}\noz\\
&\hs\ls\lf\{\sum_{k=0}^\fz 2^{-kq_-(\tz-\frac{n}r)}\lf\|\lf\{\sum_{j=1}^\fz
\lf[\frac{|\lz_j|}{\|\chi_{B_j}\|_\vp}\chi_{B_j}\r]^{p_-}\r\}^{\frac1{p_-}}\r\|_{\vp}^{q_-}
\r\}^{\frac1{q_-}}\noz\\
&\hs\sim\lf\{\sum_{k=0}^\fz 2^{-kq_-(\tz-\frac{n}r)}
\lf[\ca\lf(\{\lz_j\}_{j\in\nn},\,\{B_j\}_{j\in\nn}\r)\r]^{q_-}\r\}^{\frac1{q_-}}\noz\\
&\hs\ls\ca\lf(\{\lz_j\}_{j\in\nn},\,\{B_j\}_{j\in\nn}\r).
\end{align}
From this and \eqref{eq 5.9x}, we deduce \eqref{eq 5.xx}.

To complete the proof of Theorem \ref{thm 5}, we still need to show \eqref{eq 5.y1}.
Indeed, let $\vez\in(\frac nq,\,\fz)$, $M\in\nn$ and
$m$ be a $(p(\cdot),\,M,\,\vez)_L$-molecule associated with ball
$B:= B(x_B,\,r_B)\st\rn$ with $x_B\in\rn$ and $r_B\in(0,\,\fz)$.
Since $\frac1q-\frac12=\frac{2\az}{n}$ and $\az\in(0,\,\frac12]$,
it follows that $q\in [\frac{2n}{n+2},2)\st(p_-(L),\,2)$. Then, by Remark \ref{rem 3}, we
know that $L^{-\az}$ is bounded from $L^q(\rn)$ to $L^2(\rn)$.
From this, the boundedness of $S_L$ on $L^2(\rn)$ (see Remark \ref{rem 8}),
the H\"{o}lder inequality and the fact that $\vez\in(\frac nq,\,\fz)$,
we deduce that, when $k\in\{0,\,\ldots,\,10\}$,
\begin{align}\label{eq 5.yz}
\|S_L(L^{-\az}(m))\|_{L^2(U_k(B))}
&\ls\|m\|_{L^q(\rn)}
\sim\sum_{j=0}^\fz\|m\|_{L^q(\ujb)}\noz\\
&\ls\sum_{j=0}^\fz|2^jB|^{\frac1q-\frac12}\|m\|_{L^2(\ujb)}\noz\\
&\ls\sum_{j=0}^\fz 2^{-j(\vez-\frac{n}q)}|B|^{\frac1q}\|\chi_B\|_{\vp}^{-1}\noz\\
&\ls|B|^{\frac1q}\|\chi_B\|_{\vp}^{-1}.
\end{align}

When $k\in\zz_+\cap[11,\,\fz)$, we write
\begin{align}\label{eq 5.yy}
&\|S_L(L^{-\az}(m))\|_{L^2(U_k(B))}\noz\\
&\hs\le\lf\|S_L\lf(L^{-\az}\lf[I-e^{-r_B^2L}\r]^M(m)\r)\r\|_{L^2(U_k(B))}\noz\\
&\hs\hs+\lf\|S_L\lf(L^{-\az}\lf[I-\lf(I-e^{-r_B^2L}\r)^M\r](m)\r)\r\|_{L^2(U_k(B))}\noz\\
&\hs\ls\lf\|S_L\lf(L^{-\az}\lf[I-e^{-r_B^2L}\r]^M(m)\r)\r\|_{L^2(U_k(B))}\noz\\
&\hs\hs+\sup_{1\le l\le M}\lf\|S_L\lf(L^{-\az}\lf[\frac{l}{M}r_B^2L
e^{-\frac{l}{M}r_B^2L}\r]^M\lf(r_B^{-2}L^{-1}\r)^M(m)\r)\r\|_{L^2(U_k(B))}\noz\\
&\hs=:{\rm I}+{\rm II}.
\end{align}
To estimate the term ${\rm I}$, let $S_k(B):= (2^{k+2}B)\setminus (2^{k-3}B)$ for $k\in\zz_+\cap[11,\,\fz)$.
Then we have
\begin{align*}
{\rm I}
&\le\lf\|S_L\lf(L^{-\az}\lf[I-e^{-r_B^2L}\r]^M\lf(m\chi_{S_k(B)}\r)\r)\r\|_{L^2(U_k(B))}\\
&\hs+\lf\|S_L\lf(L^{-\az}\lf[I-e^{-r_B^2L}\r]^M\lf(m\chi_{[S_k(B)]^\com}\r)\r)\r\|_{L^2(U_k(B))}
=:{\rm I}_1+{\rm I}_2.\noz
\end{align*}
For ${\rm I}_1$, by the fact that $q\in(p_-(L),\,2)$ and \eqref{eq-od},
we find that, for any $t\in(0,\,\fz)$, $e^{-tL}$ is bounded on $L^q(\rn)$.
From this, the boundedness of $S_L$ on $L^2(\rn)$ (see \eqref{eq-1}),
the boundedness of $L^{-\az}$ from $L^q(\rn)$ to $L^2(\rn)$ and
the H\"{o}lder inequality, it follows that
\begin{align}\label{eq 5.y4}
{\rm I}_1
&\ls\lf\|L^{-\az}\lf(I-e^{-r_B^2L}\r)^M\lf(m\chi_{S_k(B)}\r)\r\|_{L^2(\rn)}\noz\\
&\ls\lf\|\lf(I-e^{-r_B^2L}\r)^M\lf(m\chi_{S_k(B)}\r)\r\|_{L^q(\rn)}
\ls\|m\|_{L^q(S_k(B))}\noz\\
&\ls\|m\|_{L^2(S_k(B))}|2^kB|^{\frac1q-\frac12}
\ls 2^{-k\vez}|2^kB|^{\frac1q}\|\chi_B\|_{\vp}^{-1}.
\end{align}
For ${\rm I}_2$, by an argument similar to that used in \cite[pp.\,774-777]{hmm11}, we conclude that
\begin{align*}
{\rm I}_2\ls 2^{-2kM}\lf(2^kr_B\r)^{2\az}\|m\|_{L^2(\rn)}.
\end{align*}
From this and Remark \ref{r-0902}, we deduce that
\begin{align*}
{\rm I}_2
\ls 2^{-2k(M-\az)}r_B^{n(\frac1q-\frac12)}|B|^{\frac12}\|\chi_B\|_{\vp}^{-1}
\ls 2^{-k(2M+\frac n2)}|2^kB|^{\frac1q}\|\chi_B\|_{\vp}^{-1}.
\end{align*}
This, together with \eqref{eq 5.y4}, implies that
\begin{align}\label{eq 5.y6}
{\rm I}\ls \lf[2^{-k\vez}+2^{-k(2M+\frac n2)}\r]|2^kB|^{\frac1q}\|\chi_B\|_{\vp}^{-1}.
\end{align}

By an argument similar to that used in the estimates of ${\rm I}$, we also obtain
\begin{align*}
{\rm II}\ls \lf[2^{-k\vez}+2^{-k(2M+\frac n2)}\r]|2^kB|^{\frac1q}\|\chi_B\|_{\vp}^{-1}.
\end{align*}

Combining this, \eqref{eq 5.y6}, \eqref{eq 5.yy} and \eqref{eq 5.yz},
we know that, for any $k\in\zz_+$ and any $(p(\cdot),\,M,\,\vez)_L$-molecule $m$,
\begin{align*}
\|S_L(L^{-\az}(m))\|_{L^2(U_k(B))}\ls 2^{-k\tz}|2^kB|^{\frac1q}\|\chi_B\|_{\vp}^{-1},
\end{align*}
where $\tz:=\min\{\vez,\,2M+\frac n2\}$. Choosing
$\vez\in(\frac{n}{p_-},\,\fz)\st(\frac nq,\,\fz)$
and $M\in\nn\cap(\frac n2[\frac1{p_-}-\frac12],\,\fz)$, we have $\tz\in(\frac{n}{p_-},\,\fz)$.
Thus, \eqref{eq 5.y1} holds true, which completes the proof of Theorem \ref{thm 5}.
\end{proof}

\begin{remark}\label{rem 6}
As a special case of Theorem \ref{thm 5}, when $p(\cdot)\equiv p$,
$q(\cdot)\equiv q$ and $\frac1p-\frac1q=\frac{2\az}{n}$,
we know that the operator $L^{-\az}$ $(\az\in(0,\,\frac12])$ is bounded from $H_L^{p}(\rn)$
to $H^{q}_L(\rn)$, which was already obtained in \cite[Theorem 7.2]{hmm11}
(see also \cite[Remark 7.3]{jy10}).
\end{remark}

\subsection{Boundedness of the Riesz transform $\nabla L^{-1/2}$\label{s5.3}}
\hskip\parindent
In this subsection, we show that the Riesz transform $\nabla L^{-1/2}$ is bounded
from $\vhp$ to the variable Hardy spaces, denoted by $H^{p(\cdot)}(\rn)$.
We begin with recalling the definition of the Riesz transform $\nabla L^{-1/2}$
and the definition of $H^{p(\cdot)}(\rn)$ introduced in
\cite{ns12} (see also \cite[Definition 3.2]{cw14}).

Let $L$ be the second order divergence form elliptic operator as in \eqref{eq op}.
The Riesz transform $\nabla L^{-1/2}$ is defined by setting, for any $f\in L^2(\rn)$
and $x\in\rn$,
\begin{align*}
\nabla L^{-1/2}(f)(x):=\frac{1}{2\sqrt{\pi}}\int_0^\fz\nabla e^{-sL}(f)(x)\frac{ds}{\sqrt{s}}.
\end{align*}
By \cite[Theorem 1.4]{ahlmt02}, we know that the domain of $L^{1/2}$ coincides
with the Sobolev space $H^1(\rn)$. Hence, for any $f\in L^2(\rn)$,
$L^{-1/2}(f)\in H^1(\rn)$ and $\nabla L^{-1/2}(f)$ stands for the
distributional derivatives of $L^{-1/2}(f)$.

Let $\cs(\rn)$ be the \emph{space of all Schwartz functions} and $\cs'(\rn)$
the \emph{space of all Schwartz distributions}. For any $N\in\nn$, define
\begin{align*}
\cf_N(\rn):=\lf\{\psi\in\cs(\rn):\ \sum_{\bz\in\zz_+^n,\,|\bz|\le N}\sup_{x\in\rn}(1+|x|)^N
\lf|D^\bz\psi(x)\r|\le 1\r\},
\end{align*}
where, for any $\bz:=(\bz_1,\,\ldots,\,\bz_n)\in\zz_+^n$,
$|\bz|:=\bz_1+\cdots +\bz_n$ and $D^\bz:=(\frac{\partial}{\partial x_1})^{\bz_1}\cdots
(\frac{\partial}{\partial x_n})^{\bz_n}$.
For any $N\in\nn$, the \emph{grand maximal function} $\cm_F$ is defined by setting,
for any $f\in\cs'(\rn)$ and $x\in\rn$,
\begin{align*}
\cm_N(f)(x):=\sup\{|\psi_t\ast f(x)|:\ t\in(0,\,\fz),\, \psi\in \cf_N(\rn)\},
\end{align*}
where, for any $t\in(0,\,\fz)$ and $\xi\in\rn$, $\psi_t(\xi):=t^{-n}\psi(\xi/t)$.

\begin{definition}\label{def hardy}
Let $p(\cdot)\in C^{\log}(\rn)$ and $N\in(\frac{n}{p_-}+n+1,\,\fz)$.
Then the \emph{variable Hardy space} $H^{p(\cdot)}(\rn)$ is defined by setting
\begin{align*}
H^{p(\cdot)}(\rn):=\lf\{f\in\cs'(\rn):\ \|f\|_{H^{p(\cdot)}(\rn)}:=\|\cm_N(f)\|_{\vp}<\fz\r\}.
\end{align*}
\end{definition}

\begin{remark}\label{rem 5}
In \cite[Theorem 3.3]{ns12}, Nakai and Sawano introduced the variable Hardy space $H^{p(\cdot)}(\rn)$ with
$p(\cdot)\in C^{\log}(\rn)$ and, in \cite[Theorem 3.3]{ns12}, proved that
the definition of $H^{p(\cdot)}(\rn)$ is independent of $N$ as long as $N$ is sufficiently large.
Independently, Cruz-Uribe and Wang \cite{cw14} also introduced and studied the variable Hardy space
$H^{p(\cdot)}(\rn)$ but with some slightly weaker assumptions on $p(\cdot)$; moreover,
in \cite[Theorem 3.1]{cw14}, they showed that the definition of $H^{p(\cdot)}(\rn)$
is independent of $N\in(\frac{n}{p_-}+n+1,\,\fz)$.
\end{remark}

An important fact of $H^{p(\cdot)}(\rn)$ is that every element in $H^{p(\cdot)}(\rn)$ admits an
atomic decomposition (see \cite{ns12,cw14}).
Let us first recall the definition of $(p(\cdot),\,q,\,s)$-atoms as follows.
Recall that, for any $s\in\rr$, $\lfloor s\rfloor$ denotes the maximal integer not more
than $s$.

\begin{definition}[\cite{ns12,cw14}]\label{def atom}
Let $p(\cdot)\in\cp(\rn)$, $q\in (p_+,\fz]\cap[1,\fz)$ and $s:=\lfloor\frac{n}{p_-}-n \rfloor$.
A measurable function $a$ on $\rn$ is called a $(p(\cdot),\,q,\,s)$-atom
associated with ball $B$ of $\rn$ if
\begin{enumerate}
\item[(i)] $\supp a\st B$;

\item[(ii)] $\|a\|_{L^q(\rn)}\le |B|^{\frac1q}\|\chi_B\|^{-1}_{\vp}$;

\item[(iii)] for any $\az\in\zz_+^n$ with $|\az|\le s$, $\int_\rn a(x)x^\az\,dx=0$.
\end{enumerate}
\end{definition}

The following lemma is just \cite[Theorem 1.1]{s13}, which
establishes the atomic decomposition of $H^{p(\cdot)}(\rn)$ (see also \cite{ns12}).

\begin{lemma}[\cite{s13}]\label{lem 5.31}
Let $p(\cdot)\in C^{\log}(\rn)$ with $p_+\in(0,\,1]$.
\begin{enumerate}
\item[{\rm (i)}] Let $q\in(p_+,\,\fz]\cap[1,\,\fz]$ and $s:=\lfloor \frac{n}{p_-}-n\rfloor$.
Then there exists a positive constant $C$ such that,
for any $\{\lz_j\}_{j\in\nn}\st\cc$ and any family $\{a_j\}_{j\in\nn}$
of $(p(\cdot),\,q,\,s)$-atoms, associated with balls $\{B_j\}_{j\in\nn}$ of $\rn$, such that
$\ca(\{\lz_j\}_{j\in\nn},\,\{B_j\}_{j\in\nn})<\fz$, it holds true that
$f:=\sum_{j\in\nn}\lz_ja_j$ converges in $H^{p(\cdot)}(\rn)$ and
$\|f\|_{H^{p(\cdot)}(\rn)}\le C\ca(\{\lz_j\}_{j\in\nn},\,\{B_j\}_{j\in\nn}).$

\item[{\rm (ii)}] Let $s\in\zz_+$. For any $f\in H^{p(\cdot)}(\rn)$,
there exists a decomposition
$f=\sum_{j=1}^\fz\lz_ja_j$ in $\cs'(\rn),$
where $\{\lz_j\}_{j\in\nn}\st\cc$ and $\{a_j\}_{j\in\nn}$ is a family
of $(p(\cdot),\,\fz,\,s)$-atoms associated with balls $\{B_j\}_{j\in\nn}$ of $\rn$.
Moreover, there exists a positive constant $C$ such that, for any $f\in H^{p(\cdot)}(\rn)$,
\begin{align*}
\ca(\{\lz_j\}_{j\in\nn},\,\{B_j\}_{j\in\nn})\le C\|f\|_{H^{p(\cdot)}(\rn)}.
\end{align*}
\end{enumerate}
\end{lemma}

\begin{remark}
We point out that, in \cite{cw14}, Cruz-Uribe and Wang also established the
atomic characterizations of $H^{p(\cdot)}(\rn)$. However, the atomic characterization of
$H^{p(\cdot)}(\rn)$ obtained in \cite{cw14} is quite different from that of
the classical atomic characterization (and also that of \cite[Theorem 1.1]{s13}), which
was based on the atomic characterization established by Str\"omberg and Torchinsky \cite{st89}
for weighted Hardy spaces.
\end{remark}

The following proposition is an analogue of \cite[Proposition 4.7]{zcjy15}
(see also \cite{bckyy13,jy11}), its proof being omitted.
\begin{proposition}\label{pro 5}
Let $p(\cdot)\in\cp(\rn)$ with $\frac{n}{n+1}<p_-\le p_+\le 1$ and $\vez\in(0,\,\fz)$.
Suppose that $m\in L^2(\rn)$ is a function satisfying $\int_\rn m(x)\,dx=0$ and
there exists a ball $B\st\rn$ such that, for any $j\in\zz_+$,
$\|m\|_{L^2(\ujb)}\le 2^{-j\vez}|2^jB|^{\frac12}\|\chi_B\|_{\vp}^{-1}.$
Then
\begin{align*}
m=\wz{C}\lf(\sum_{j=1}^\fz 2^{-j\vez}\az_j\r)\ \ \ \text{in}\ \ L^2(\rn),
\end{align*}
where $\{\az_j\}_{j\in\nn}$ is a family of $(p(\cdot),\,2,\,0)$-atoms associated with balls
$\{2^{j+1}B\}_{j\in\nn}$ and $\wz{C}$ a positive constant independent of $m$.
\end{proposition}

To establish the boundedness of $\nabla L^{-1/2}$ on $\vhp$, we also need
the following technical lemma, which was proved in \cite[Theorem 3.4]{hm09}.
\begin{lemma}[\cite{hm09}]\label{lem 5.1}
Let $L$ be the second order divergence form elliptic operator as in \eqref{eq op}.
Then there exist positive constants $C$ and $M\in\nn$ with $M>n/4$ such that, for any $t\in(0,\,\fz)$,
closed subsets $E,\,F\st\rn$ with $\dist(E,\,F)>0$ and $f\in L^2(\rn)$ with $\supp f\st E$,
\begin{align*}
\lf\|\nabla L^{-1/2}\lf(I-e^{-tL}\r)^M(f)\r\|_{L^2(F)}
\le C\lf(\frac{t}{[\dist(E,\,F)]^2}\r)^M\|f\|_{L^2(E)}
\end{align*}
and
\begin{align*}
\lf\|\nabla L^{-1/2}\lf(tLe^{-tL}\r)^M(f)\r\|_{L^2(F)}
\le C\lf(\frac{t}{[\dist(E,\,F)]^2}\r)^M\|f\|_{L^2(E)}.
\end{align*}
\end{lemma}

\begin{theorem}\label{thm 6}
Let $p(\cdot)\in C^{\log}(\rn)$ with $\frac{n}{n+1}<p_-\le p_+\le 1$ and $L$ be
the second order divergence form elliptic operator as in \eqref{eq op}. Then there exists
a positive constant $C$ such that, for any $f\in\vhp$,
\begin{align}\label{eq 5.z1}
\lf\|\nabla L^{-1/2}(f)\r\|_{H^{p(\cdot)}(\rn)}\le C\|f\|_{\vhp}.
\end{align}
\end{theorem}

\begin{proof}
Since $\vhp\cap L^2(\rn)$ is dense in $\vhp$, to prove Theorem \ref{thm 6}, we only need to
show that \eqref{eq 5.z1} holds true for all $f\in\vhp\cap L^2(\rn)$.

By Proposition \ref{pro 2}, we find that, for any $f\in\vhp\cap L^2(\rn)$,
$M\in\nn$ and $\vez\in(0,\,\fz)$,
there exist $\{\lz_j\}_{j\in\nn}\st\cc$ and a family $\{m_j\}_{j\in\nn}$ of
$(p(\cdot),\,M,\,\vez)_L$-molecules associated with balls $\{B_j\}_{j\in\nn}$ of $\rn$ such that
\begin{align}\label{eq 5.z2}
f=\sum_{j=1}^\fz\lz_jm_j\ \ \ \text{in}\ \ L^2(\rn)
\end{align}
and
\begin{align}\label{eq 5.z5}
\ca(\{\lz_j\}_{j\in\nn},\,\{B_j\}_{j\in\nn})\ls \|f\|_{\vhp}.
\end{align}
From \eqref{eq 5.z2}, the boundedness of $\nabla L^{-1/2}$ on $L^2(\rn)$
(see \cite[Theorem 1.4]{ahlmt02}) and Riesz theorem,
we deduce that
\begin{align}\label{eq 5.z3}
\nabla L^{-1/2}(f)=\sum_{j=1}^\fz\lz_j\nabla L^{-1/2}(m_j)\ \ \ \text{in}\ \ L^2(\rn).
\end{align}
Here and hereafter, for any $g\in L^2(\rn)$, let
$$\nabla L^{-1/2}(g):=\lf(\frac{\partial}{\partial x_1}L^{-1/2}(g),\,\ldots,\,
\frac{\partial}{\partial x_n}L^{-1/2}(g)\r)=:
(\partial_1 L^{-1/2}(g),\,\ldots,\,\partial_n L^{-1/2}(g)).$$

Let $M\in\nn\cap(\frac n2[\frac1{p_-}-\frac12],\,\fz)$ and
$\vez\in(\frac{n}{p_-},\,\fz)$.
Next, we show that, for any $(p(\cdot),\,M,\,\vez)_L$-molecule
$m$, associated with ball $B:= B(x_B,\,r_B)\st\rn$ with $x_B\in\rn$
and $r_B\in(0,\,\fz)$, and $j\in\zz_+$,
\begin{align}\label{eq 5.7}
\lf\|\lf[\sum_{l=1}^n\lf|\partial_l L^{-1/2}(m)\r|^2\r]^{\frac12}\r\|_{L^2(\ujb)}
&=:\lf\|\nabla L^{-1/2}(m)\r\|_{L^2(\ujb)}\noz\\
&\ls 2^{-j\tz}|2^jB|^{\frac12}\|\chi_B\|_{\vp}^{-1},
\end{align}
where $\tz\in(\frac{n}{p_-},\,\fz)$.

Indeed, when $j\in\{0,\,\ldots,\,10\}$, from the boundedness of $\nabla L^{-1/2}$ on $L^2(\rn)$
(see \cite[Theorem 1.4]{ahlmt02}) and Remark \ref{r-0902}, it follows that
\begin{align*}
\lf\|\nabla L^{-1/2}(m)\r\|_{L^2(\ujb)}\ls \|m\|_{L^2(\rn)}\ls|B|^{\frac12}\|\chi_B\|_{\vp}^{-1}.
\end{align*}
When $j\in\zz_+\cap[11,\,\fz)$, we write
\begin{align}\label{eq 5.8}
&\lf\|\nabla L^{-1/2}(m)\r\|_{L^2(\ujb)}\noz\\
&\hs\le\lf\|\nabla L^{-1/2}\lf(I-e^{-r_B^2L}\r)^M(m)\r\|_{L^2(\ujb)}\noz\\
&\hs\hs+\lf\|\nabla L^{-1/2}\lf[I-\lf(I-e^{-r_B^2L}\r)^M\r](m)\r\|_{L^2(\ujb)}\noz\\
&\hs\ls\lf\|\nabla L^{-1/2}\lf(I-e^{-r_B^2L}\r)^M(m)\r\|_{L^2(\ujb)}\noz\\
&\hs\hs+\sup_{1\le k\le M}\lf\|\nabla L^{-1/2}\lf(\frac{k}{M}r_B^2L e^{-\frac{k}{M}r_B^2L}\r)^M
\lf(r_B^{-2}L^{-1}\r)^M(m)\r\|_{L^2(\ujb)}\noz\\
&\hs=:{\rm I}+{\rm II}.
\end{align}

We first estimate ${\rm I}$. For any $j\in\zz_+\cap[11,\,\fz)$, let
$S_j(B):=(2^{j+1}B)\setminus (2^{j-2}B)$. It is easy to see that
$\dist([S_j(B)]^\com,\,\ujb)\sim 2^jr_B$. From this, the boundedness of $\nabla L^{-1/2}$
on $L^2(\rn)$ (see \cite[Theorem 1.4]{ahlmt02}),
Lemma \ref{lem 5.1} and Remark \ref{r-0902}, we deduce that
\begin{align}\label{eq 5.6}
{\rm I}
&\le\lf\|\nabla L^{-1/2}\lf(I-e^{-r_B^2L}\r)^M\lf(m\chi_{S_j(B)}\r)\r\|_{L^2(\ujb)}\noz\\
&\hs+\lf\|\nabla L^{-1/2}\lf(I-e^{-r_B^2L}\r)^M\lf(m\chi_{[S_j(B)]^\com}\r)\r\|_{L^2(\ujb)}\noz\\
&\ls\|m\|_{L^2(S_j(B))}+\lf(\frac{r_B}{2^jr_B}\r)^{2M}\|m\|_{L^2(\rn)}\noz\\
&\ls\lf[2^{-j\vez}+2^{-j(2M+\frac{n}2)}\r]|2^jB|^{\frac12}\|\chi_B\|_{\vp}^{-1}.
\end{align}
By an argument similar to that used in the proof of \eqref{eq 5.6}, we have
\begin{align*}
{\rm II}
&\le\sup_{1\le k\le M}\lf\|\nabla L^{-1/2}\lf(\frac{k}{M}r_B^2L e^{-\frac{k}{M}r_B^2L}\r)^M
\lf[\lf(r_B^{-2}L^{-1}\r)^M(m)\chi_{S_j(B)}\r]\r\|_{L^2(\ujb)}\\
&\hs+\sup_{1\le k\le M}\lf\|\nabla L^{-1/2}\lf(\frac{k}{M}r_B^2L e^{-\frac{k}{M}r_B^2L}\r)^M
\lf[\lf(r_B^{-2}L^{-1}\r)^M(m)\chi_{[S_j(B)]^\com}\r]\r\|_{L^2(\ujb)}\\
&\ls\lf\|\lf(r_B^{-2}L^{-1}\r)(m)\r\|_{L^2(S_j(B))}+\lf(\frac{r_B}{2^jr_B}\r)^{2M}
\lf\|\lf(r_B^{-2}L^{-1}\r)(m)\r\|_{L^2(\rn)}\\
&\ls\lf[2^{-j\vez}+2^{-j(2M+\frac{n}2)}\r]|2^jB|^{\frac12}\|\chi_B\|_{\vp}^{-1}.
\end{align*}
This, together with \eqref{eq 5.6} and \eqref{eq 5.8}, implies
\eqref{eq 5.7} with
$$\tz:=\min\lf\{\vez,\,2M+\frac{n}2\r\}\in\lf(\frac{n}{p_-},\,\fz\r).$$

On the other hand, by an argument similar to that used in
the proof of \cite[Theorem 7.4]{jy10}, we know that,
for any $(p(\cdot),\,M,\,\vez)_L$-molecule $m$ and $l\in\{1,\,\ldots,n\}$,
$$\int_\rn \partial_l L^{-1/2}(m)(x)\,dx=0.$$
From this, \eqref{eq 5.7}, Proposition \ref{pro 5} and \eqref{eq 5.z3}, it follows that,
for any $l\in\{1,\,\ldots,\,n\}$,
\begin{align}\label{eq 5.z4}
\partial_l L^{-1/2}(f)= \wz{C}\lf(\sum_{j=1}^\fz\sum_{k=1}^\fz\lz_j
2^{-k\tz}\az_{j,\,k}\r)\ \ \ \text{in}\ \ L^2(\rn),
\end{align}
where $\{\az_{j,\,k}\}_{j,\,k\in\nn}$ is a family of $(p(\cdot),\,2,\,0)$-atoms
associated with balls $\{2^{k+1}B_j\}_{j,\,k\in\nn}$ and
$\wz C$ is a positive constant independent of $f$.
Noticing that $p_-\in(\frac{n}{n+1},\,1]$, we then know that
$s:=\lfloor n(\frac{1}{p_-}-1)\rfloor=0$.
From this, Lemma \ref{lem 5.31}(i), \eqref{eq 5.z4}, Remark \ref{rem 1},
an argument similar to that used in \eqref{eq 5.y3x} and \eqref{eq 5.y5x},
the fact that $\tz\in(\frac{n}{p_-},\,\fz)$ and \eqref{eq 5.z5},
we deduce that, for any $l\in\{1,\,\ldots,\,n\}$ and $f\in\vhp\cap L^2(\rn)$,
\begin{align*}
\lf\|\partial_l L^{-1/2}(f)\r\|_{H^{p(\cdot)}(\rn)}
&\ls\ca\lf(\lf\{\lz_j2^{-k\tz}\r\}_{j,\,k\in\nn},\,\lf\{2^{k+1}B_j\r\}_{j,\,k\in\nn}\r)\\
&\sim\lf\|\sum_{k=1}^\fz 2^{-k\tz p_-}\sum_{j=1}^\fz\lf[\frac{|\lz_j|\chi_{2^{k+1}B_j}}
{\|\chi_{2^{k+1}B_j}\|_{\vp}}\r]^{p_-}\r\|_{L^{\frac{p(\cdot)}{p_-}}(\rn)}^{\frac1{p_-}}\\
&\ls\lf\{\sum_{k=1}^\fz 2^{-k\tz p_-}\lf\|\lf\{\sum_{j=1}^\fz\lf[\frac{|\lz_j|\chi_{2^{k+1}B_j}}
{\|\chi_{B_j}\|_{\vp}}\r]^{p_-}\r\}^{\frac1{p_-}}\r\|_{L^{p(\cdot)}(\rn)}
^{p_-}\r\}^{\frac1{p_-}}\\
&\ls\lf\{\sum_{k=1}^\fz 2^{-k(\tz-\frac nr) p_-}
\lf[\ca(\{\lz_j\}_{j\in\nn},\,\{B_j\}_{j\in\nn})\r]^{p_-}\r\}^{\frac1{p_-}}\\
&\ls\ca(\{\lz_j\}_{j\in\nn},\,\{B_j\}_{j\in\nn})\ls\|f\|_{\vhp},
\end{align*}
where $r\in(0,\,p_-)$ such that $\tz>\frac{n}r$.
Therefore, \eqref{eq 5.z1} holds true for any $f\in\vhp\cap L^2(\rn)$,
which completes the proof of Theorem \ref{thm 6}.
\end{proof}

\begin{remark}\label{rem 7}
When $p(\cdot)\equiv p$ with $p\in(\frac{n}{n+1},\,1]$, Theorem \ref{thm 6}
was established in \cite[Proposition 5.6]{hmm11} (see also \cite[Theorem 7.4]{jy10}).
\end{remark}

\subsection*{Acknowledgements}
The authors would like to thank Professor Sibei Yang
for some helpful conversations on this topic.
They would also like to express their deep thanks to the referee
for her/his very carefully reading and several constructive and stimulating
comments which lead to some essential improvements of the main results of this article.


\bigskip

\noindent{Dachun Yang and Junqiang Zhang (Corresponding author)}

\medskip

\noindent{\small School of Mathematical Sciences, Beijing Normal University,
Laboratory of Mathematics and Complex Systems, Ministry of Education,
Beijing 100875, People's Republic of China}

\smallskip

\noindent{\it E-mails}: \texttt{dcyang@bnu.edu.cn} (D. Yang)

\hspace{1.12cm}\texttt{zhangjunqiang@mail.bnu.edu.cn} (J. Zhang)

\bigskip

\noindent Ciqiang Zhuo

\medskip

\noindent{\small Key Laboratory of High Performance Computing and Stochastic
Information Processing (HPCSIP) (Ministry of Education of China), College of Mathematics
and Computer Science, Hunan Normal University, Changsha, Hunan 410081, People's Republic of China}

\smallskip

\noindent {\it E-mail}: \texttt{cqzhuo@mail.bnu.edu.cn} (C. Zhuo)

\end{document}